\documentclass[12pt, letterpaper]{article}
\usepackage{amsmath}
\usepackage{amsfonts}
\usepackage{amssymb}
\usepackage{amsthm}
\usepackage{breqn}
\usepackage{setspace}
\usepackage{fullpage}
\usepackage{enumitem}
\usepackage{comment}
\usepackage{hyperref}
\usepackage{bbm}
\usepackage{tikz}
\usepackage{mathtools}
\usepackage[utf8]{inputenc}
\usepackage[english]{babel}
\usetikzlibrary{patterns,arrows,decorations.pathreplacing}
\newtheorem{theorem}{Theorem}
\newtheorem{lemma}[theorem]{Lemma}
\newtheorem{corollary}[theorem]{Corollary}
\newtheorem{definition}[theorem]{Definition}

\newtheorem{conjecture}[theorem]{Conjecture}

\newtheorem{remark}{Remark}
\newtheorem{proposition}[theorem]{Proposition}
\newtheorem{case}{Case}

\def \circ {14pt}
\def \redstar {\textbf{\Large\color{red}{$\star$}}}

\DeclareMathOperator{\ex}{ex}

\usepackage{authblk}
\title{Planar Tur\'an Number of the 6-Cycle}
\author[1]{Debarun Ghosh}
\author[1,2]{Ervin Gy\H{o}ri}
\author[3]{Ryan R. Martin}
\author[1]{\newline Addisu Paulos}
\author[1]{Chuanqi Xiao}
\affil[1]{Central European University, Budapest\par
\texttt{ chuanqixm@gmail.com, ghosh\textunderscore debarun@phd.ceu.edu,addisu\textunderscore 2004@yahoo.com}}
\affil[2]{Alfr\'ed R\'enyi Institute of Mathematics, Budapest \par
\texttt{gyori.ervin@renyi.mta.hu}}
\affil[3]{Iowa State University, Ames, IA, USA \par
\texttt{rymartin@iastate.edu}}
\date{}

\begin{document}
\maketitle
\baselineskip=0.30in
\begin{abstract}
    Let ${\rm ex}_{\mathcal{P}}(n,T,H)$ denote the maximum number of copies of $T$ in an $n$-vertex planar graph which does not contain $H$ as a subgraph. When $T=K_2$, ${\rm ex}_{\mathcal{P}}(n,T,H)$ is the well studied function, the planar Tur\'an number of $H$, denoted by ${\rm ex}_{\mathcal{P}}(n,H)$. The topic of extremal planar graphs was initiated by Dowden (2016). He obtained sharp upper bound for both ${\rm ex}_{\mathcal{P}}(n,C_4)$ and ${\rm ex}_{\mathcal{P}}(n,C_5)$. Later on, Y. Lan, et al. continued this topic and proved that ${\rm ex}_{\mathcal{P}}(n,C_6)\leq \frac{18(n-2)}{7}$. In this paper, we give a sharp upper bound ${\rm ex}_{\mathcal{P}}(n,C_6) \leq \frac{5}{2}n-7$, for all $n\geq 18$, which improves Lan's result. We also pose a conjecture on ${\rm ex}_{\mathcal{P}}(n,C_k)$, for $k\geq 7$.
\end{abstract}
{\bf Keywords}\ \ Planar Tur\'an number, Extremal planar graph

\section{Introduction and Main Results}
In this paper, all graphs considered are planar, undirected, finite and contain neither loops nor multiple edges. We use $C_k$ to denote the cycle on $k$ vertices and $K_r$ to denote the complete graph on $r$ vertices.

One of the well-known results in extremal graph theory is the Tur\'an Theorem \cite{TP}, which gives the maximum number of edges that a graph on $n$ vertices can have without containing a $K_r$ as a subgraph. The Erd\H{o}s-Stone-Simonovits Theorem \cite{EP1,EP2} then generalized this result and asymptotically determines $\ex(n,H)$ for all non-bipartite graphs $H$: $\ex(n,H)=(1-\frac{1}{\chi(H)-1})\binom{n}{2}+o(n^{2})$, where $\chi(H)$ denotes the chromatic number of $H$. Over the last decade, a considerable amount of research work has been carried out in Tur\'an-type problems, i.e., when host graphs are $K_n$, $k$-uniform hypergraphs or $k$-partite graphs, see \cite{EP2,zykov}.

In 2016, Dowden \cite{DC} initiated the study of Tur\'an-type problems when host graphs are planar, i.e., how many edges can a planar graph on $n$ vertices have, without containing a given smaller graph? The planar Tur\'an number of a graph $H$, $\ex_{\mathcal{P}}(n,H)$, is the maximum number of edges in a planar graph on $n$ vertices which does not contain $H$ as a subgraph.  Dowden \cite{DC} obtained the tight bounds  $\ex_{\mathcal{P}}(n,C_4) \leq\frac{15(n-2)}{7}$, for all $n\geq 4$ and $\ex_{\mathcal{P}}(n,C_5) \leq\frac{12n-33}{5}$, for all $n\geq 11$.  Later on, Y. Lan, et al.~\cite{LY} obtained bounds $\ex_{\mathcal{P}}(n,\Theta_4)\leq \frac{12(n-2)}{5}$, for all $n\geq 4$, $\ex_{\mathcal{P}}(n,\Theta_5)\leq \frac{5(n-2)}{2}$, for all $n\geq 5$ and
$\ex_{\mathcal{P}}(n,\Theta_6)\leq \frac{18(n-2)}{7}$, for all $n\geq 7$, where $\Theta_k$ is obtained from a cycle $C_k$ by adding an additional edge joining any two non-consecutive vertices. They also demonstrated that their bounds for $\Theta_4$ and $\Theta_5$ are tight by showing infinitely many values of $n$ and planar graph on $n$ vertices attaining the stated bounds. As a consequence of the bound for $\Theta_6$ in the same paper, they presented the following corollary.

\begin{corollary}[Y. Lan, et al.\cite{LY}]
    \begin{align*}
        \ex_{\mathcal{P}}(n, C_6)\leq \frac{18(n-2)}{7}
    \end{align*}
    for all $n\geq 6$, with equality when $n=9$.
\end{corollary}
In this paper we present a tight bound for $\ex_{\mathcal{P}}(n, C_6)$. In particular, we prove the following two theorems to give the tight bound.

We denote the vertex and the edge sets of a graph $G$ by $V(G)$ and $E(G)$ respectively. We also denote the number of vertices and edges of $G$ by $v(G)$ and $e(G)$ respectively. The minimum degree of $G$ is denoted $\delta(G)$. The main ingredient of the result is as follows:
\begin{theorem}\label{thm:main}
    Let $G$ be a $2$-connected, $C_6$-free plane graph on $n$ $(n\geq 6)$ vertices with $\delta(G)\geq 3$. Then $e(G) \leq \frac{5}{2} n - 7$.
\end{theorem}

We use Theorem~\ref{thm:main}, which considers only $2$-connected graphs with no degree $2$ (or $1$) vertices and order at least $6$, in order to establish our desired result, which bounds gives the desired bound of $\frac{5}{2}n-7$ for all $C_6$-free plane graphs with at least $18$ vertices.
\begin{theorem}\label{thm:main_new}
    Let $G$ be a $C_6$-free plane graph on $n$ $(n\geq 18)$ vertices. Then
    \begin{align*}
        e(G) \leq \frac{5}{2} n - 7 .
    \end{align*}
\end{theorem}

Indeed, there are $17$-vertex graphs on $17$ vertices with $36$ edges, but $\frac{5}{2}(17)-7=35.5<36$. One such graph can be seen in Figure~\ref{fig:smallexample}.

\begin{figure}[ht]
    \centering
    \begin{tikzpicture}[scale=0.2]
        \coordinate (w1) at (-12,0);
        \coordinate (w2) at (-8,8);
        \coordinate (w3) at (-4,0);
        \coordinate (w4) at (-8,2);
        \coordinate (w5) at (-8,5);
        \coordinate (x1) at (-4,0);
        \coordinate (x2) at (0,8);
        \coordinate (x3) at (4,0);
        \coordinate (x4) at (0,2);
        \coordinate (x5) at (0,5);
        \coordinate (y1) at (4,0);
        \coordinate (y2) at (8,8);
        \coordinate (y3) at (12,0);
        \coordinate (y4) at (8,2);
        \coordinate (y5) at (8,5);
        \coordinate (z1) at (12,0);
        \coordinate (z2) at (16,8);
        \coordinate (z3) at (20,0);
        \coordinate (z4) at (16,2);
        \coordinate (z5) at (16,5);
        \draw[very thick] (w1) -- (w2) -- (w3) -- (w1);
        \draw[very thick] (w1) -- (w4) -- (w3);
        \draw[very thick] (w1) -- (w5) -- (w3);
        \draw[very thick] (w2) -- (w5) -- (w4);
        \draw[very thick] (x1) -- (x2) -- (x3) -- (x1);
        \draw[very thick] (x1) -- (x4) -- (x3);
        \draw[very thick] (x1) -- (x5) -- (x3);
        \draw[very thick] (x2) -- (x5) -- (x4);
        \draw[very thick] (y1) -- (y2) -- (y3) -- (y1);
        \draw[very thick] (y1) -- (y4) -- (y3);
        \draw[very thick] (y1) -- (y5) -- (y3);
        \draw[very thick] (y2) -- (y5) -- (y4);
        \draw[very thick] (z1) -- (z2) -- (z3) -- (z1);
        \draw[very thick] (z1) -- (z4) -- (z3);
        \draw[very thick] (z1) -- (z5) -- (z3);
        \draw[very thick] (z2) -- (z5) -- (z4);
        \draw[fill=black] (w1) circle(\circ);
        \draw[fill=black] (w2) circle(\circ);
        \draw[fill=black] (w3) circle(\circ);
        \draw[fill=black] (w4) circle(\circ);
        \draw[fill=black] (w5) circle(\circ);
        \draw[fill=black] (x1) circle(\circ);
        \draw[fill=black] (x2) circle(\circ);
        \draw[fill=black] (x3) circle(\circ);
        \draw[fill=black] (x4) circle(\circ);
        \draw[fill=black] (x5) circle(\circ);
        \draw[fill=black] (y1) circle(\circ);
        \draw[fill=black] (y2) circle(\circ);
        \draw[fill=black] (y3) circle(\circ);
        \draw[fill=black] (y4) circle(\circ);
        \draw[fill=black] (y5) circle(\circ);
        \draw[fill=black] (z1) circle(\circ);
        \draw[fill=black] (z2) circle(\circ);
        \draw[fill=black] (z3) circle(\circ);
        \draw[fill=black] (z4) circle(\circ);
        \draw[fill=black] (z5) circle(\circ);
    \end{tikzpicture}
    \caption{Example of $G$ on $17$ vertices such that $e(G) > (5/2) v(G) - 7$.}
    \label{fig:smallexample}
\end{figure}
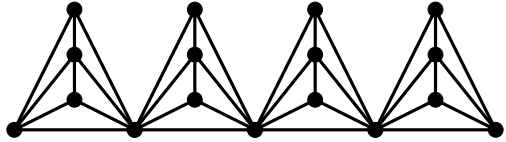

We show that, for large graphs, Theorem~\ref{thm:main_new} is tight:
\begin{theorem}\label{thm:construct}
    For every $n\cong 2\pmod{5}$, there exists a $C_6$-free plane graph $G$ with $v(G) = \frac{18n+14}{5}$ and $e(G) = 9n$, hence $e(G) = \frac{5}{2} v(G) - 7$.
\end{theorem}

For a vertex $v$ in $G$, the neighborhood of $v$, denoted $N_G(v)$, is the set of all vertices in $G$ which are adjacent to $v$. We denote the degree of $v$ by $d_G(v) = |N_G(v)|$. We may avoid the subscripts if the underlying graph is clear. The minimum degree of $G$ is denoted by $\delta(G)$, the number of components of $G$ is denoted by $c(G)$. For the sake of simplicity, we may use the term $k$-cycle to mean a cycle of length $k$ and $k$-face to mean a face bounded by a $k$-cycle. A $k$-path is a path with $k$ edges.

\section{Proof of Theorem~\ref{thm:construct}: Extremal Graph Construction}

First we show that for a plane graph $G_0$ with $n$ vertices ($n\cong 7\pmod{10}$), each face having length $7$ and each vertex in $G_0$ having degree either $2$ or $3$, we can construct $G$, where $G$ is a $C_6$-free plane graph with $v(G) = \frac{18n+14}{5}$ and $e(G) = 9n$. We then give a construction for such a $G_0$ as long as $n\cong 7\pmod{10}$.

Using Euler's formula, the fact that every face has length $7$ and every degree is $2$ or $3$, we have $e(G_0)=\frac{7(n-2)}{5}$ and the number of degree $2$ and degree $3$ vertices in $G_0$ are $\frac{n+28}{5}$ and  $\frac{4n-28}{5}$, respectively.

Given $G_0$, we construct first an intermediate graph $G'$ by step \ref{it:construct1}:
\begin{enumerate}[label=(\arabic*)]
    \item Add halving vertices to each edge of $G_0$ and join the pair of halving vertices with distance $2$, see an example in Figure~\ref{fig:op1}. Let $G'$ denote this new graph, then $v(G')=v(G_0)+e(G_0)=\frac{12n-14}{5}$ and the number of degree $2$ and degree $3$ vertices in $G'$ is equal to the number of degree $2$ and degree $3$ vertices in $G_0$, respectively. \label{it:construct1}

    \begin{figure}[ht]
        \def \radius {200pt}
        \centering
        \begin{tikzpicture}[scale=0.2]
            \def \n {7}
            \foreach \s in {1,...,\n}
            {
                \def \here {{360/\n * (\s - 1) + 90}}
                \def \next {{360/\n * (\s) + 90}}
                \draw[fill=black] (\here : \radius) circle(\circ);
                \draw[very thick] (\here : \radius) -- (\next : \radius);
            }
        \end{tikzpicture}
        \begin{tikzpicture}[scale=0.2]
            \node at (0,0){$\Longrightarrow$};
            \node at (-4,-6) {};
            \node at (4,0) {};
        \end{tikzpicture}
        \begin{tikzpicture}[scale=0.2]
            \def \n {7}
            \def \inradius {180pt}
           \foreach \s in {1,...,\n}
           {
                \def \here {{360/\n * (\s - 1) + 90}}
                \def \next {{360/\n * (\s) + 90}}
                \def \half {{360/\n * (\s - 1/2) + 90}}
                \def \nexthalf {{360/\n * (\s + 1/2) + 90}}
                \draw[fill=black] (\here : \radius) circle(\circ);
                \draw[very thick] (\here : \radius) -- (\next : \radius);
                \draw[red,fill=red] (\half : \inradius) circle(\circ);
                \draw[red,very thick] (\half : \inradius) -- (\nexthalf : \inradius);
            }
        \end{tikzpicture}
        \caption{Adding a halving vertex to each edge of $G_0$.}
        \label{fig:op1}
    \end{figure}
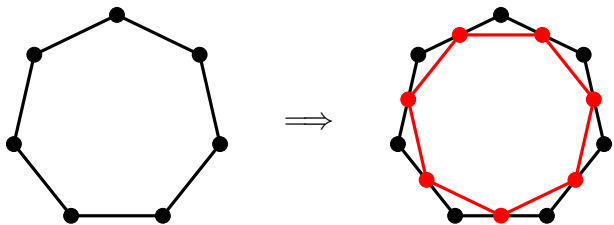

To get $G$, we apply the following steps \ref{it:construct2} and \ref{it:construct3} on the degree $2$ and $3$ vertices in $G'$, respectively.

    \item For each degree $2$ vertex $v$ in $G_0$, let $N(v)=\{v_1,v_2\}$, and so $v_1vv_2$ forms an induced triangle in $G'$. Fix $v_1$ and $v_2$, replace $v_1vv_2$ with a $K^{-}_5$ by adding vertices $v^{'}_1$, $v^{'}_2$ to $V(G')$ and edges $v^{'}_1v$, $v^{'}_1v^{'}_2$, $v^{'}_1v_1$, $v^{'}_1v_2$, $v^{'}_2v_1$, $v^{'}_2v_2$ to $E(G')$. See Figure~\ref{fig:op2}. \label{it:construct2}
    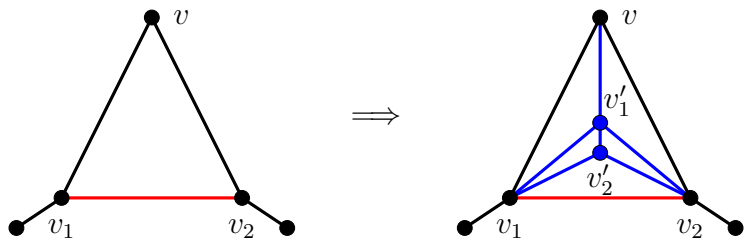
\begin{figure}[ht]
        \centering
        \begin{tikzpicture}[scale=0.2]
            \coordinate (v) at (0,12);
            \coordinate (v1) at (-6,0);
            \coordinate (v2) at (6,0);
            \coordinate (leafl) at (-9,-2);
            \coordinate (leafr) at (9,-2);
            \coordinate (leaft) at (0,15);
            \draw[red,very thick] (v1) -- (v2);
            \draw[very thick] (leafl) -- (v1) -- (v) -- (v2) -- (leafr);
            \draw[fill=black] (v1) circle(\circ)  node[label=below:$v_1$] {};
            \draw[fill=black] (v2) circle(\circ)  node[label=below:$v_2$] {};
            \draw[fill=black] (v) circle(\circ)  node[label=right:$v$] {};
            \draw[fill=black] (leafl) circle(\circ);
            \draw[fill=black] (leafr) circle(\circ);
            \node at (leaft) {}; 
        \end{tikzpicture}
        \begin{tikzpicture}[scale=0.2]
            \node at (0,0){$\Longrightarrow$};
            \node at (-4,-8) {};
            \node at (4,0) {};
        \end{tikzpicture}
        \begin{tikzpicture}[scale=0.2]
            \coordinate (v) at (0,12);
            \coordinate (v1) at (-6,0);
            \coordinate (v2) at (6,0);
            \coordinate (v1p) at (0,5);
            \coordinate (v2p) at (0,3);
            \coordinate (leafl) at (-9,-2);
            \coordinate (leafr) at (9,-2);
            \coordinate (leaft) at (0,15);
            \draw[red,very thick] (v1) -- (v2);
            \draw[blue,very thick] (v1) -- (v2p) -- (v2);
            \draw[blue,very thick] (v1) -- (v1p) -- (v2);
            \draw[blue,very thick] (v2p) -- (v1p) -- (v);
            \draw[very thick] (leafl) -- (v1) -- (v) -- (v2) -- (leafr);
            \draw[fill=black] (v1) circle(\circ)  node[label=below:$v_1$] {};
            \draw[fill=black] (v2) circle(\circ)  node[label=below:$v_2$] {};
            \draw[fill=black] (v) circle(\circ)  node[label=right:$v$] {};
            \draw[fill=black] (leafl) circle(\circ);
            \draw[fill=black] (leafr) circle(\circ);
            \draw[thin,black,fill=blue] (v1p) circle(\circ)  node[label={[label distance=-8pt]30:$v_1'$}] {};
            \draw[thin,black,fill=blue] (v2p) circle(\circ)  node[label={[label distance=-4pt]270:$v_2'$}] {};
            \node at (leaft) {}; 
        \end{tikzpicture}
        \caption{Replacing a degree-$2$ vertex of $G_0$ with a $K_5^-$.}
        \label{fig:op2}
    \end{figure}

    \item For each degree $3$ vertex $v$ in $G_0$, such that $N(v)=\{v_1,v_2,v_3\}$, the set of vertices $\{v,v_1,v_2,v_3\}$ then forms an induced $K_4$ in $G'$. Fix $v_1$, $v_2$ and $v_3$, replace this $K_4$ with a $K^{-}_5$ by adding a new vertex $v'$ to $V(G')$ and edges $v'v$, $v'v_1$, $v'v_2$ to $E(G')$. See Figure~\ref{fig:op3}.\label{it:construct3}
    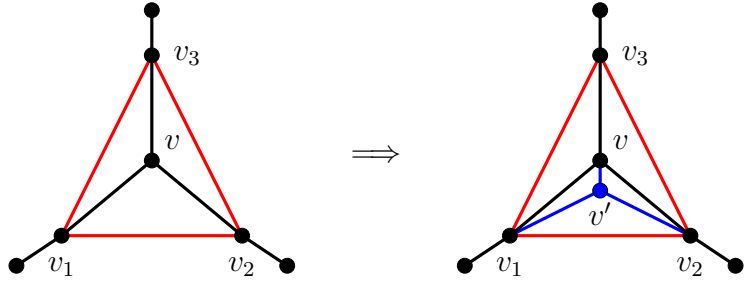
\begin{figure}[ht]
        \centering
        \begin{tikzpicture}[scale=0.2]
            \coordinate (v1) at (-6,0);
            \coordinate (v2) at (6,0);
            \coordinate (v3) at (0,12);
            \coordinate (v) at (0,5);
            \coordinate (leafl) at (-9,-2);
            \coordinate (leafr) at (9,-2);
            \coordinate (leaft) at (0,15);
            \draw[red,very thick] (v1) -- (v2) -- (v3) -- (v1);
            \draw[very thick] (leafl) -- (v1) -- (v) -- (v2) -- (leafr);
            \draw[very thick] (leaft) -- (v3) -- (v);
            \draw[fill=black] (v1) circle(\circ)  node[label=below:$v_1$] {};
            \draw[fill=black] (v2) circle(\circ)  node[label=below:$v_2$] {};
            \draw[fill=black] (v3) circle(\circ)  node[label=right:$v_3$] {};
            \draw[fill=black] (v) circle(\circ)  node[label={[label distance=-4pt]30:$v$}] {};
            \draw[fill=black] (leafl) circle(\circ);
            \draw[fill=black] (leafr) circle(\circ);
            \draw[fill=black] (leaft) circle(\circ);
        \end{tikzpicture}
        \begin{tikzpicture}[scale=0.2]
            \node at (0,0){$\Longrightarrow$};
            \node at (-4,-8) {};
            \node at (4,0) {};
        \end{tikzpicture}
        \begin{tikzpicture}[scale=0.2]
            \coordinate (v1) at (-6,0);
            \coordinate (v2) at (6,0);
            \coordinate (v3) at (0,12);
            \coordinate (v) at (0,5);
            \coordinate (vp) at (0,3);
            \coordinate (leafl) at (-9,-2);
            \coordinate (leafr) at (9,-2);
            \coordinate (leaft) at (0,15);
            \draw[red,very thick] (v1) -- (v2) -- (v3) -- (v1);
            \draw[blue,very thick] (v1) -- (vp) -- (v2);
            \draw[blue,very thick] (vp) -- (v);
            \draw[very thick] (leafl) -- (v1) -- (v) -- (v2) -- (leafr);
            \draw[very thick] (leaft) -- (v3) -- (v);
            \draw[fill=black] (v1) circle(\circ)  node[label=below:$v_1$] {};
            \draw[fill=black] (v2) circle(\circ)  node[label=below:$v_2$] {};
            \draw[fill=black] (v3) circle(\circ)  node[label=right:$v_3$] {};
            \draw[fill=black] (v) circle(\circ)  node[label={[label distance=-4pt]30:$v$}] {};
            \draw[fill=black] (leafl) circle(\circ);
            \draw[fill=black] (leafr) circle(\circ);
            \draw[fill=black] (leaft) circle(\circ);
            \draw[thin,black,fill=blue] (vp) circle(\circ)  node[label={[label distance=-4pt]270:$v'$}] {};
        \end{tikzpicture}
        \caption{Replacing a degree-$3$ vertex of $G_0$ with a $K_5^-$.}
        \label{fig:op3}
    \end{figure}
\end{enumerate}

For each integer $k\geq 0$, and $n = 10 k + 7$ we present a construction for such a $G_0$, call it $G_0^k$:
Let $v_{i}^{t}$ and $v_{i}^{b}$ $(1\leq i\leq k+1)$ be the top and bottom vertices of the heptagonal grids with $3$ layers and $k$ columns, respectively (see the red vertices in Figure~\ref{fig:heptgrid}) and $v$ be the extra vertex in $G_0^k$ but not in the heptagonal grid. We join $v_1^tv$, $vv_1^b$ and $v_{i}^{t}v_{i}^{b}$ $(2\leq i\leq k+1)$. Clearly, $G_0^k$ is a $(10k+7)$-vertex plane graph and each face of $G_0^k$ is a $7$-face. Obviously $e\left(G_0^k\right)=14k+7$, and the number of degree $2$ and $3$ vertices are $2 k + 7 = \frac{n + 28}{5}$ and $8 k = \frac{4 n - 28}{5}$ respectively.

\begin{figure}[ht]
    \centering
    \begin{tikzpicture}[xscale=0.66,yscale=0.55]
        \clip (-10.5,-8.5) rectangle (12.5, 8); 
        \def \myellipse{ellipse (5.0pt and 6.0pt)} 
        \draw[fill=black] (-5,5) \myellipse;
        \foreach \s in {-4,-2,0,2,4} \draw[fill=black] (\s,4) \myellipse;
        \foreach \s in {-2,0,2,4,6} \draw[fill=black] (\s,2) \myellipse;
        \foreach \s in {1.5,3.5,5.5} \draw[fill=black] (\s,1.5) \myellipse;
        \foreach \s in {-3,-1,1,3,5,7} \draw[fill=black] (\s,1) \myellipse;
        \foreach \s in {-3,-1,1,3,5,7} \draw[fill=black] (\s,-1) \myellipse;
        \foreach \s in {-4,-2,0,2,4,6,8} \draw[fill=black] (\s,-2) \myellipse;
        \foreach \s in {-2,0,2,4,6,8} \draw[fill=black] (\s,-4) \myellipse;
        \draw[fill=black] (7,-5) \myellipse;
        \foreach \s in {-2.5,-0.5,3.5,5.5} \draw[fill=black] (\s,-4.5) \myellipse;
        \foreach \s in {-1,1}
        {
            \foreach \t in {-5,-3,1,3,5}
            {
                \draw[very thick] ({\t-\s},{3*\s+1}) -- ({\t-\s+1},{3*\s+2}) -- ({\t-\s+2},{3*\s+1});
                \draw[very thick] ({\t-\s},{3*\s-1}) -- ({\t-\s+1},{3*\s-2}) -- ({\t-\s+2},{3*\s-1});
                \draw[very thick] ({\t-\s},{3*\s+1}) -- ({\t-\s},{3*\s-1});
                \draw[very thick] (\t,1) -- (\t,-1);
            }
            \draw[very thick] (-\s-1,3*\s+1) -- (-\s-1,3*\s-1);
            \draw[very thick] (-\s+7,3*\s+1) -- (-\s+7,3*\s-1);
            \draw[very thick] (-1,1) -- (-1,-1);
            \draw[very thick] (7,1) -- (7,-1);
        }
        \draw[very thick] (-2,2) -- (-1,1);
        \draw[very thick] (6,2) -- (7,1);
        \draw[very thick] (-5,-1) -- (-4,-2);
        \draw[very thick] (1,-1) -- (2,-2);
        \draw[very thick] (-1,3) node {$\cdots$};
        \draw[very thick] (0,0) node {$\cdots$};
        \draw[very thick] (1,-3) node {$\cdots$};
        \coordinate (v1t) at (-6,4);
        \coordinate (v2t) at (-3,5);
        \coordinate (vkm2t) at (1,5);
        \coordinate (vkm1t) at (3,5);
        \coordinate (vkt) at (5,5);
        \coordinate (vkp1t) at (6,4);
        \coordinate (v1b) at (-4,-4);
        \coordinate (v2b) at (-3,-5);
        \coordinate (v3b) at (-1,-5);
        \coordinate (vkm1b) at (3,-5);
        \coordinate (vkb) at (5,-5);
        \coordinate (vkp1b) at (7.5,-4.5);
        \coordinate (v) at (-8.1,-0.9);
        \coordinate (x1) at (-6,2);
        \coordinate (x2) at (-5,1);
        \coordinate (x3) at (-5,-1);
        \coordinate (x4) at (-4.5,1.5);
        \coordinate (x5) at (-4,2);
        \coordinate (y) at (-2.5,1.5);
        \draw[very thick] (v1t) to[out=180,in=-165,distance=125pt] (v1b);
        \draw[very thick] (v2t) to[out=145,in=-165,distance=300pt] (v2b);
        \draw[very thick] (vkp1t) to[out=35,in=0,distance=135pt] (vkp1b);
        \draw[very thick] (vkt) to[out=25,in=-35,distance=275pt] (vkb);
        \draw[very thick] (vkm1t) to[out=40,in=-45,distance=475pt] (vkm1b);
        \draw[very thick,fill=black] (v) \myellipse node[label=left:$v$] {};
        \draw[thin,fill=red] (v1t) \myellipse node[label={[label distance=-3pt]90:$v_1^t$}] {};
        \draw[thin,fill=red] (v2t) \myellipse node[label={[label distance=-3pt]90:$v_2^t$}] {};
        \draw[thin,fill=red] (vkm2t) \myellipse node[label={[label distance=-3pt]90:$v_{k-2}^t$}] {};
        \draw[thin,fill=red] (vkm1t) \myellipse node[label={[label distance=-3pt]90:$v_{k-1}^t$}] {};
        \draw[thin,fill=red] (vkt) \myellipse node[label={[label distance=-3pt]90:$v_k^t$}] {};
        \draw[thin,fill=red] (vkp1t) \myellipse node[label={[label distance=-5pt]300:$v_{k+1}^t$}] {};
        \draw[thin,fill=red] (v1b) \myellipse node[label={[label distance=-5pt]120:$v_1^b$}] {};
        \draw[thin,fill=red] (v2b) \myellipse node[label={[label distance=-2pt]270:$v_2^b$}] {};
        \draw[thin,fill=red] (v3b) \myellipse node[label={[label distance=-2pt]270:$v_3^b$}] {};
        \draw[thin,fill=red] (vkm1b) \myellipse node[label={[label distance=-2pt]270:$v_{k-1}^b$}] {};
        \draw[thin,fill=red] (vkb) \myellipse node[label={[label distance=-3pt]270:$v_k^b$}] {};
        \draw[thin,fill=red] (vkp1b) \myellipse node[label={[label distance=-6pt]275:$v_{k+1}^b$}] {};
        \draw[thin,fill=green] (x1) \myellipse node[label={[label distance=-6pt]195:$x_1$}] {};
        \draw[thin,fill=green] (x5) \myellipse node[label={[label distance=-6pt]15:$x_5$}] {};
        \draw[thin,fill=green] (x4) \myellipse node[label={[label distance=-6pt]315:$x_4$}] {};
        \draw[thin,fill=blue] (y) \myellipse node[label={[label distance=-6pt]315:$y$}] {};
        \draw[thin,fill=green] (x2) \myellipse node[label={[label distance=-6pt]195:$x_2$}] {};
        \draw[thin,fill=green] (x3) \myellipse node[label={[label distance=-6pt]195:$x_3$}] {};
    \end{tikzpicture}
    \caption{The graph $G_0^k$, $k\geq 1$, in which each face has length $7$. The graph $H_0^k$ (see Remark~\ref{rem:H0k}) is obtained by deleting $x_1,\ldots,x_5$ and adding the edge $v_1^ty$.}
    \label{fig:heptgrid}
\end{figure}
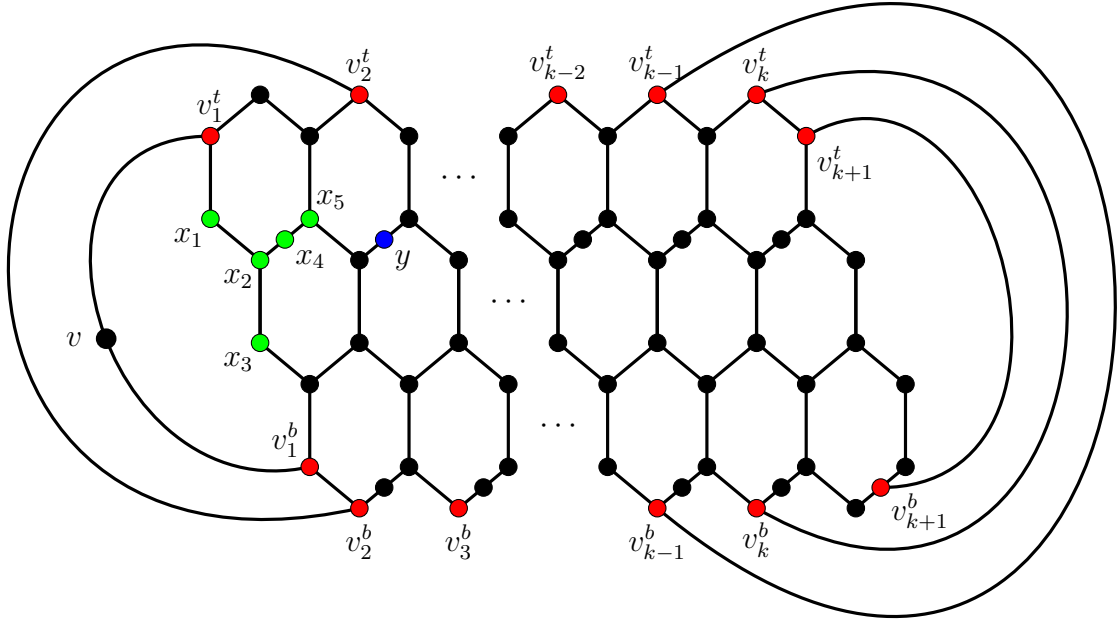

After applying steps \ref{it:construct1}, \ref{it:construct2}, and \ref{it:construct3} on $G_0^k$, we get $G$.
It is easy to verify that $G$ is a $C_6$-free plane graph with
\begin{align*}
    v(G) &= v(G_0^k) + e(G_0^k) + 2 (2 k + 7) + 8 k = (10 k + 7) + (14 k + 7) + 12 k + 14 &&= 36 k + 28 \\
    e(G) &= 9 v(G_0^k)= 90k+63 .
\end{align*}
Thus, $e(G)=\frac{5}{2}v(G)-7$.

\begin{remark}\label{rem:H0k}
    In fact, for $k\geq 1$ and $n=10k+2$, there exists a graph $H^k_0$ which is obtained from $G^k_0$ by deleting vertices (colored green in Figure~\ref{fig:heptgrid}) $x_1$, $x_2$, $x_3$, $x_4$, $x_5$ and adding the edge $v^t_1y$. Clearly, $H^k_0$ is an $10k+2$-vertex plane graph such that all faces have length $7$. Moreover, $e(H^k_0) = 14 k$, the number of degree-$2$ and degree-$3$ vertices are $2 k + 6 = \frac{n + 28}{5}$ and $8 k - 4 = \frac{4 n - 28}{5}$, respectively. After applying steps (1), (2), and (3) to $H^k_0$, we get a graph $H$ that is a $C_6$-free plane graph with $e(H) = (5/2) v(H) - 7$.

    Thus, for any $k\cong 2\pmod{5}$, we have the graphs above such that each face is a 7-gon and we get a $C_6$-free plane graph on $n$ vertices with $(5/2)n - 7$ edges for  $n\cong 10\pmod{18}$ if $n\geq 28$.
\end{remark}

\section{Definitions and Preliminaries}

We give some necessary definitions and preliminary results which are needed in the proof of Theorems~\ref{thm:main} and~\ref{thm:main_new}.

\begin{definition}
    Let $G$ be a plane graph and $e\in E(G)$.  If $e$ is not in a $3$-face of $G$, then we call it a \textbf{trivial triangular-block}.  Otherwise, we recursively construct a \textbf{triangular-block} in the following way.  Start with $H$ as a subgraph of $G$, such that $E(H)=\{e\}$.
    \begin{enumerate}[label=(\arabic*)]
        \item Add the other edges of the $3$-face containing $e$ to $E(H)$.
        \item Take $e'\in E(H)$ and search for a $3$-face containing $e'$. Add these other edge(s) in this $3$-face to $E(H)$.
        \item Repeat step (2) till we cannot find a $3$-face for any edge in $E(H)$.
    \end{enumerate}
    We denote the triangular-block obtained from $e$ as the starting edge, by $B(e)$.
\end{definition}

Let $G$ be a plane graph. We have the following three observations:
\begin{enumerate}[label=(\roman*)]
    \item If $H$ is a non-trivial triangular-block and $e_1, e_2\in E(H)$, then $B(e_1)=B(e_2)=H$.
    \item Any two triangular-blocks of $G$ are edge disjoint. \label{it:obs2}
    \item If $B$ is a triangular-block with the unbounded region being a $3$-face, then $B$ is a triangulation graph.
\end{enumerate}

Let $\mathcal{B}$ be the family of triangular-blocks of $G$. From observation \ref{it:obs2} above, we have
\begin{align*}
    e(G)=\sum\limits_{B\in\mathcal{B}}e(B),
\end{align*}
where $e(G)$ and $e(B)$ are the number of edges of $G$ and $B$ respectively.

Next, we distinguish the types of triangular-blocks that a $C_6$-free plane graph may contain. The following lemma gives us the bound on the number of vertices of triangular-blocks.
\begin{lemma}\label{lem:5block}
    Every triangular-block of $G$ contains at most $5$ vertices.
\end{lemma}
\begin{proof}
    We prove it by contradiction. Let $B$ be a triangular-block of $G$ containing at least $6$ vertices. We perform the following operations: delete one vertex from the boundary of the unbounded face of $B$ sequentially until the number of vertices of the new triangular block $B^{'}$ is $6$.  Next, we show that $B'$ is not a triangular-block in $G$. Suppose that it is. We consider the following two cases to complete the proof.

    \begin{case} $B'$ contains a separating triangle.\end{case}
        Let $v_1v_2v_3$ be the separating triangle. Without loss of generality, assume that the inner region of the triangle contains two vertices say, $v_4$ and $v_5$. The outer region of the triangle contains one vertex, say $v_6$. Since the unbounded face is a $3$-face, the inner structure is a triangulation. Without loss of generality, let the inner structure be as shown in Figure \ref{fig:blocksize}(a). Now consider the vertex $v_6$. If $v_1,v_2\in N(v_6)$, then $v_3v_4v_5v_2v_6v_1v_3$ is a $6$-cycle in $G$, a contradiction. Similarly for the cases when $v_1,v_3\in N(v_6)$ and $v_2,v_3\in N(v_6)$.

    \begin{case} $B'$ contains no separating triangle.\end{case}
        Consider a triangular face $v_1v_2v_3v_1$. Let $v_4$ be a vertex in the triangular-block such that $v_2v_3v_4v_2$ is a $3$-face. Notice that $v_1v_4 \notin E(B')$, otherwise we get a separating triangle in $B'$. Let $v_5$ be a vertex in $B'$ such that $v_2v_4v_5v_2$ is a $3$-face. Notice that $v_6$ cannot be adjacent to both vertices in any of the pairs $\{v_1,v_2\}$, $\{v_1,v_3\}$, $\{v_2,v_5\}$, $\{v_3,v_4\}$, or $\{v_4,v_5\}$. Otherwise, $C_6\subset G$. Also $v_3v_5 \notin E(B')$, otherwise we have a separating triangle. So, let $v_1v_5\in E(B')$ and $v_1,v_5\in N(v_6)$ (see Figure~\ref{fig:blocksize}(b)). In this case $v_1v_6v_5v_2v_4v_3v_1$ results in a $6$-cycle, a contradiction.
\end{proof}

        \begin{figure}[ht]
            \centering
            \begin{tikzpicture}[scale=0.2]
            \coordinate (v2) at (-8,0);
            \coordinate (v1) at (0,16);
            \coordinate (v3) at (8,0);
            \coordinate (v4) at (0,4);
            \coordinate (v5) at (0,7);
            \draw[very thick] (v2) -- (v1) -- (v3) -- (v2);
            \draw[very thick] (v2) -- (v4) -- (v3);
            \draw[very thick] (v2) -- (v5) -- (v3);
            \draw[very thick] (v1) -- (v5) -- (v4);
            \draw[fill=black] (v2) circle(\circ)  node[label=below:$v_2$] {};
            \draw[fill=black] (v1) circle(\circ)  node[label=right:$v_1$] {};
            \draw[fill=black] (v3) circle(\circ)  node[label=below:$v_3$] {};
            \draw[fill=black] (v5) circle(\circ)  node[label={[label distance=-1pt]0:$v_5$}] {};
            \draw[fill=black] (v4) circle(\circ)  node[label={[label distance=-1pt]270:$v_4$}] {};
            \node at (0,-5){(a)};
        \end{tikzpicture}
        \qquad\qquad\qquad\qquad
        \begin{tikzpicture}[scale=0.2]
            \coordinate (v1) at (-8,7);
            \coordinate (v2) at (0,7);
            \coordinate (v3) at (-4,0);
            \coordinate (v4) at (4,0);
            \coordinate (v5) at (8,7);
            \coordinate (v6) at (0,16);
            \draw[very thick] (v1) -- (v3) -- (v4) -- (v5) -- (v6) -- (v1);
            \draw[very thick] (v1) -- (v2) -- (v5);
            \draw[very thick] (v3) -- (v2) -- (v4);
            \draw[very thick] (v1) to[out=45,in=135,distance=250pt] (v5);
            \draw[fill=black] (v1) circle(\circ)  node[label=above:$v_1$] {};
            \draw[fill=black] (v2) circle(\circ)  node[label=above:$v_2$] {};
            \draw[fill=black] (v3) circle(\circ)  node[label=below:$v_3$] {};
            \draw[fill=black] (v4) circle(\circ)  node[label=below:$v_4$] {};
            \draw[fill=black] (v5) circle(\circ)  node[label=above:$v_5$] {};
            \draw[fill=black] (v6) circle(\circ)  node[label=right:$v_6$] {};
            \node at (0,-5){(b)};
        \end{tikzpicture}
        \caption{The structure of $B'$ when it contains a separating triangle or not, respectively.}
        \label{fig:blocksize}
    \end{figure}
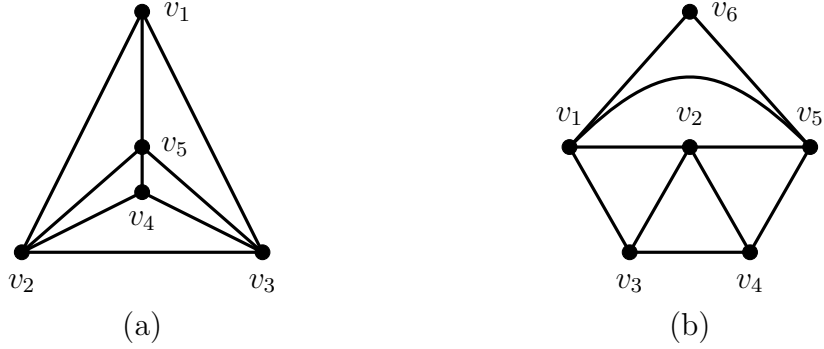

Now we describe all possible triangular-blocks in $G$ based on the number of vertices the block contains. For $k\in\{2,3,4,5\}$, we denote the triangular-blocks on $k$ vertices as $B_k$.

\subsubsection*{Triangular-blocks on $5$ vertices.}
There are four types of triangular-blocks on $5$ vertices (see Figure~\ref{fig:5blocks}). Notice that $B_{5,a}$ is a $K_5^-$ .
\begin{figure}[ht]
    \centering
    \begin{tikzpicture}[scale=0.2]
        \coordinate (x1) at (-4,0);
        \coordinate (x2) at (0,8);
        \coordinate (x3) at (4,0);
        \coordinate (x4) at (0,2);
        \coordinate (x5) at (0,5);
        \draw[very thick] (x1) -- (x2) -- (x3) -- (x1);
        \draw[very thick] (x1) -- (x4) -- (x3);
        \draw[very thick] (x1) -- (x5) -- (x3);
        \draw[very thick] (x2) -- (x5) -- (x4);
        \draw[fill=black] (x1) circle(\circ);
        \draw[fill=black] (x2) circle(\circ);
        \draw[fill=black] (x3) circle(\circ);
        \draw[fill=black] (x4) circle(\circ);
        \draw[fill=black] (x5) circle(\circ);
        \node at (0,-3){$B_{5,a}$};
    \end{tikzpicture}
    \qquad\qquad
    \begin{tikzpicture}[scale=0.2]
        \coordinate (x1) at (-4,0);
        \coordinate (x2) at (4,0);
        \coordinate (x3) at (4,8);
        \coordinate (x4) at (-4,8);
        \coordinate (x5) at (0,4);
        \draw[very thick] (x1) -- (x2) -- (x3) -- (x4) -- (x1);
        \draw[very thick] (x1) -- (x5);
        \draw[very thick] (x2) -- (x5);
        \draw[very thick] (x3) -- (x5);
        \draw[very thick] (x4) -- (x5);
        \draw[fill=black] (x1) circle(\circ);
        \draw[fill=black] (x2) circle(\circ);
        \draw[fill=black] (x3) circle(\circ);
        \draw[fill=black] (x4) circle(\circ);
        \draw[fill=black] (x5) circle(\circ);
        \node at (0,-3){$B_{5,b}$};
    \end{tikzpicture}
    \qquad\qquad
    \begin{tikzpicture}[scale=0.2]
        \coordinate (x1) at (-4,0);
        \coordinate (x2) at (4,0);
        \coordinate (x3) at (7,4);
        \coordinate (x4) at (4,8);
        \coordinate (x5) at (-4,8);
        \draw[very thick] (x1) -- (x2) -- (x3) -- (x4) -- (x5) -- (x1);
        \draw[very thick] (x1) -- (x4) -- (x2);
        \draw[fill=black] (x1) circle(\circ);
        \draw[fill=black] (x2) circle(\circ);
        \draw[fill=black] (x3) circle(\circ);
        \draw[fill=black] (x4) circle(\circ);
        \draw[fill=black] (x5) circle(\circ);
        \node at (0,-3){$B_{5,c}$};
    \end{tikzpicture}
    \qquad\qquad
    \begin{tikzpicture}[scale=0.2]
        \coordinate (x1) at (0,0);
        \coordinate (x2) at (4,4);
        \coordinate (x3) at (0,8);
        \coordinate (x4) at (-4,4);
        \coordinate (x5) at (2,4);
        \draw[very thick] (x1) -- (x2) -- (x3) -- (x4) -- (x1);
        \draw[very thick] (x1) -- (x3);
        \draw[very thick] (x1) -- (x5) -- (x3);
        \draw[very thick] (x2) -- (x5);
        \draw[fill=black] (x1) circle(\circ);
        \draw[fill=black] (x2) circle(\circ);
        \draw[fill=black] (x3) circle(\circ);
        \draw[fill=black] (x4) circle(\circ);
        \draw[fill=black] (x5) circle(\circ);
        \node at (0,-3){$B_{5,d}$};
    \end{tikzpicture}
    \caption{Triangular-blocks on $5$ vertices.}
    \label{fig:5blocks}
\end{figure}
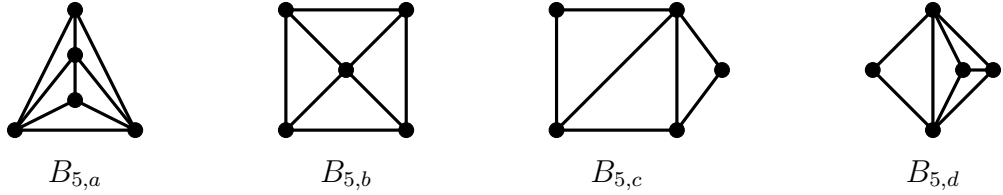

\subsubsection*{Triangular-blocks on $4$, $3$, and $3$ vertices.}
There are two types of triangular-blocks on $4$ vertices. See Figure~\ref{fig:432blocks}. Observe that $B_{4,a}$ is a $K_4$. The $3$-vertex and $2$-vertex triangular-blocks are simply $K_3$ and $K_2$ (the trivial triangular-block), respectively.
\begin{figure}[ht]
    \centering
    \begin{tikzpicture}[scale=0.2]
        \coordinate (x1) at (-4,0);
        \coordinate (x2) at (4,0);
        \coordinate (x3) at (0,8);
        \coordinate (x4) at (0,3);
        \draw[very thick] (x1) -- (x2) -- (x3) -- (x1);
        \draw[very thick] (x1) -- (x4);
        \draw[very thick] (x2) -- (x4);
        \draw[very thick] (x3) -- (x4);
        \draw[fill=black] (x1) circle(\circ);
        \draw[fill=black] (x2) circle(\circ);
        \draw[fill=black] (x3) circle(\circ);
        \draw[fill=black] (x4) circle(\circ);
        \node at (0,-3){$B_{4,a}$};
    \end{tikzpicture}
    \qquad\qquad
    \begin{tikzpicture}[scale=0.2]
        \coordinate (x1) at (0,0);
        \coordinate (x2) at (4,4);
        \coordinate (x3) at (0,8);
        \coordinate (x4) at (-4,4);
        \draw[very thick] (x1) -- (x2) -- (x3) -- (x4) -- (x1);
        \draw[very thick] (x2) -- (x4);
        \draw[fill=black] (x1) circle(\circ);
        \draw[fill=black] (x2) circle(\circ);
        \draw[fill=black] (x3) circle(\circ);
        \draw[fill=black] (x4) circle(\circ);
        \node at (0,-3){$B_{4,b}$};
    \end{tikzpicture}
    \qquad\qquad
    \begin{tikzpicture}[scale=0.2]
        \coordinate (x1) at (-4,0);
        \coordinate (x2) at (4,0);
        \coordinate (x3) at (0,8);
        \draw[very thick] (x1) -- (x2) -- (x3) -- (x1);
        \draw[fill=black] (x1) circle(\circ);
        \draw[fill=black] (x2) circle(\circ);
        \draw[fill=black] (x3) circle(\circ);
        \node at (0,-3){$B_3$};
    \end{tikzpicture}
    \qquad\qquad
    \begin{tikzpicture}[scale=0.2]
        \coordinate (x1) at (-4,4);
        \coordinate (x2) at (4,4);
        \draw[very thick] (x1) -- (x2);
        \draw[fill=black] (x1) circle(\circ);
        \draw[fill=black] (x2) circle(\circ);
        \node at (0,-3){$B_2$};
    \end{tikzpicture}
    \caption{Triangular-blocks on 4,3 and 2 vertices.}
    \label{fig:432blocks}
\end{figure}
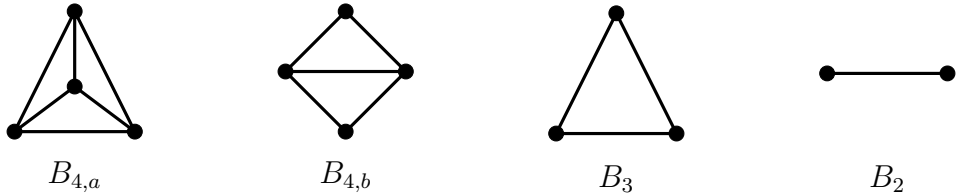

\begin{definition} Let $G$ be a plane graph.
    \begin{enumerate}[label=(\roman*)]
        \item A vertex $v$ in $G$ is called a \textbf{junction vertex} if it is in at least two distinct triangular-blocks of $G$.
        \item Let $B$ be a triangular-block in $G$. An edge of $B$ is called an \textbf{exterior edge} if it is on a boundary of non-triangular face of $G$. Otherwise, we call it an \textbf{interior edge}. An endvertex of an exterior edge is called an \textbf{exterior vertex}. We denote the set of all exterior and interior edges of $B$ by $Ext(B)$ and $Int(B)$ respectively. Let $e\in Ext(B)$, a non-triangular face of $G$ with $e$ on the boundary is called the \textbf{exterior face} of $e$.
    \end{enumerate}
\end{definition}

Notice that an exterior edge of a non-trivial triangular-block has exactly one exterior face. On the other hand,  if $G$ is a $2$-connected plane graph, then every trivial triangular-block has two exterior faces.
For a non-trivial triangular-block $B$ of a plane graph $G$, we call a path $P=v_1v_2v_3\dots v_k$ an \textit{exterior path} of $B$, if $v_1$ and $v_k$ are junction vertices and $v_iv_{i+1}$ are exterior edges of $B$ for $i\in\{1,2,\dots,k-1\}$ and $v_j$ is not junction vertex for all $j\in\{2,3,\dots,k-1\}$. The corresponding face in $G$ where $P$ is on the boundary of the face is called the \textit{exterior face} of $P$.

Next, we give the definition of the contribution of a vertex and an edge to the number of vertices and faces of $C_6$-free plane graph $G$. All graphs discussed from now on are $C_6$-free plane graph.
\begin{definition}
Let $G$ be a plane graph, $B$ be a triangular-block in $G$ and $v\in V(B)$. The contribution of $v$ to the vertex number of $B$ is denoted by $n_B(v)$, and is defined as
\begin{equation*}
    n_B(v) = \dfrac{1}{\#~\mbox{triangular-blocks in $G$ containing $v$}}.
\end{equation*}
We define the contribution of $B$ to the number of vertices of $G$ as $n(B)=\sum\limits_{v\in V(B)}n_B(v)$.
\end{definition}
Obviously, $v(G)=\sum\limits_{B\in \mathcal{B}}n(B)$, where $v(G)$ is the number of vertices in $G$ and  $\mathcal{B}$ is the family of triangular-blocks of $G$.

Let $B_{K^{-}_5}$ be a triangular-block of $G$ isomorphic to a $B_{5,a}$ with exterior vertices $v_1,v_2,v_3$, where $v_1$ and $v_3$ are junction vertices, see Figure \ref{fig:11to9} for an example. Let $F$ be a face in $G$ such that $V(F)$ contains all exterior vertices $v_{1,1}, \dots, v_{1,m}, v_{2,1}, \dots, v_{2,m}, v_{3,1}, \dots, v_{3,m}$ of $m$ $(m\geq 1)$ copies of $B_{K^-_5}$, such that $v_{1,i}, v_{2,i}, v_{3,i}$ are the exterior vertices of the $i$-th $B_{K^{-}_5}$ and $v_{1,i}$, $v_{3,i}$ $(1\leq i\leq m)$ are junction vertices.  Let $C_F$ denote the cycle associated with the face $F$.  We alter $E(C_F)$ in the following way:
\begin{align*}
    E(C_F'):=E(C_F)-\{v_{1,1}v_{2,1}v_{3,1}\}-\dots-\{v_{1,m}v_{2,m}v_{3,m}\}\cup\{v_{1,1}v_{3,1}\}\cup \dots \cup\{v_{1,m},v_{3,m}\}.
\end{align*}
Hence, the length of $F$ as $|E(C_F')|=|E(C_F)|-m$. For example, in Figure \ref{fig:11to9}, $|E(C_F)|=11$ but $|E(C_F')|=9$.

\begin{figure}[ht]
    \centering
    \begin{tikzpicture}[scale=0.2]
        \def \n {9};
        \def \radius {15};
        \coordinate (v1) at ({(360/\n) * (4.5) + 90}:{\radius-13});
        \coordinate (v2) at ({(360/\n) * (4) + 90}:\radius);
        \coordinate (v3) at ({(360/\n) * (5) + 90}:\radius);
        \coordinate (v4) at ({(360/\n) * (4.5) + 90}:{\radius-4});
        \coordinate (v5) at ({(360/\n) * (4.5) + 90}:{\radius-6.5});
        \coordinate (leafl) at ({(360/\n) * (3.75) + 90}:{\radius+5});
        \coordinate (leafr) at ({(360/\n) * (5.25) + 90}:{\radius+5});
        \draw[red,very thick] (leafl) -- (v2);
        \draw[red,very thick] (leafr) -- (v3);
        \draw[fill=black] (leafl) circle(\circ);
        \draw[fill=black] (leafr) circle(\circ);
        \draw[very thick] (v2) -- (v1) -- (v3);
        \draw[very thick] (v2) -- (v4) -- (v3);
        \draw[very thick] (v2) -- (v5) -- (v3);
        \draw[very thick] (v1) -- (v5) -- (v4);
        \draw[fill=black] (v1) circle(\circ) node[label=right:$v_1$] {};
        \draw[fill=black] (v2) circle(\circ) node[label=below:$v_2$] {};
        \draw[fill=black] (v3) circle(\circ) node[label=below:$v_3$] {};
        \draw[fill=black] (v4) circle(\circ) node[label={[label distance=-2pt]270:$v_4$}] {};
        \draw[fill=black] (v5) circle(\circ) node[label={[label distance=-9pt]60:$v_5$}] {};
        \coordinate (w1) at ({(360/\n) * (0.5) + 90}:{\radius-13});
        \coordinate (w2) at ({(360/\n) * (0) + 90}:\radius);
        \coordinate (w3) at ({(360/\n) * (1) + 90}:\radius);
        \coordinate (w4) at ({(360/\n) * (0.5) + 90}:{\radius-4});
        \coordinate (w5) at ({(360/\n) * (0.5) + 90}:{\radius-7});
        \draw[very thick] (w2) -- (w1) -- (w3);
        \draw[very thick] (w2) -- (w4) -- (w3);
        \draw[very thick] (w2) -- (w5) -- (w3);
        \draw[very thick] (w1) -- (w5) -- (w4);
        \draw[fill=black] (w1) circle(\circ);
        \draw[fill=black] (w4) circle(\circ);
        \draw[fill=black] (w5) circle(\circ);
        \foreach \s in {1,...,\n}
        {
            \def \here {{(360/\n) * (\s - 1) + 90}}
            \def \next {{(360/\n) * (\s) + 90}}
            \draw[very thick] (\here : \radius) -- (\next : \radius);
           \draw[fill=black] (\here : \radius) circle(\circ);
        }
    \end{tikzpicture}
    \caption{An example of a face containing all the exterior vertices of at least one $B_{K^-_5}$.}
    \label{fig:11to9}
\end{figure}
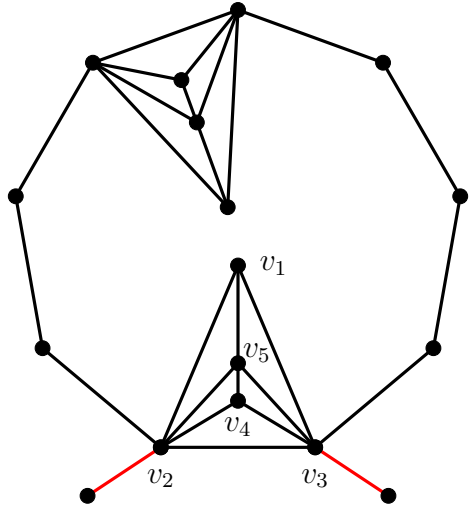

Now we are able to define the \textbf{contribution} of an ``\textit{edge}'' to the number of faces of $C_6$-free plane graph $G$.
\begin{definition}\label{bvvc}
    Let $F$ be a exterior face of $G$ and $C_F:=\{e_1,e_2,\dots, e_k\}$ be the cycle associated with $F$. The contribution of an exterior edge $e$ to the face number of the exterior face $F$, is denoted by $f_F(e)$, and is defined as follows.
    \begin{enumerate}[label=(\roman*)]
        \item If $e_1$ and $e_2$ are adjacent exterior edges of $B_{K_5^-}$, then $f_F(e_1)+f_F(e_2)=\dfrac{1}{|C_F'|}$, and $f_F(e_i)=\dfrac{1}{|C_F'|}, \text{ where }i\in\{3,4,\dots,k\}$.
        \item Otherwise,  $f_F(e)=\dfrac{1}{|C_F|}$.
    \end{enumerate}
\end{definition}

Note that $\sum\limits_{e\in E(F)}f_F(e)=1$.  For a triangular-block $B$, the total face contribution of $B$ is denoted by $f_B$ and defined as $f_B=(\#~\mbox{interior faces of $B$}) + \sum\limits_{e\in Ext(B)}f_F(e)$, where $F$ is the exterior face of $B$ with
respective to $e$.
Obviously, $f(G)=\sum\limits_{B\in\mathcal{B}}f(B)$, where $f(G)$ is the number of faces of $G$.

\section{Proof of Theorem~\ref{thm:main}}
We begin by outlining our proof. Let $f$, $n$, and $e$ be the number of faces, vertices, and edges of $G$ respectively. Let $\mathcal{B}$ be the family of all triangular-blocks of $G$.

The main target of the proof is to show that
\begin{align}
    7 f + 2 n - 5 e \leq 0\label{eq:main}.
\end{align}

Once we show \eqref{eq:main}, then by using Euler's Formula, $e=f+n-2$, we can finish the proof of Theorem~\ref{thm:main}. To prove \eqref{eq:main}, we show the existence of a partition $\mathcal{P}_1, \mathcal{P}_2,\dots,\mathcal{P}_m$ of $\mathcal{B}$ such that $7\sum\limits_{B\in \mathcal{P}_i}f(B)+2\sum\limits_{B\in\mathcal{P}_i}n(B)-5\sum\limits_{B\in \mathcal{P}_i}e(B)\leq 0$, for all $i\in\{1,2,3\dots,m\}$. Since $f=\sum\limits_{B\in \mathcal{B}}f(B)$, $n=\sum\limits_{B\in\mathcal{B}}n(B)$ and $e=\sum\limits_{B\in \mathcal{B}}e(B)$ we have
\begin{align*}
    7 f + 2 n - 5 e
    &= 7\sum\limits_{i}^{m}\sum\limits_{B\in\mathcal{P}_i}f(B) + 2\sum\limits_{i}^{m}\sum\limits_{B\in\mathcal{P}_i}n(B) - 5\sum\limits_{i}^{m}\sum\limits_{B\in\mathcal{P}_i}e(B)\\
    &= \sum\limits_{i}^{m}\bigg(7\sum\limits_{B\in\mathcal{P}_i}f(B) + 2\sum\limits_{B\in\mathcal{P}_i}n(B) - 5\sum\limits_{B\in\mathcal{P}_i}e(B)\bigg)\leq 0.
\end{align*}

The following proposition will be useful in many lemmas.
\begin{proposition}\label{prop:faces}
    Let $G$ be a $2$-connected, $C_6$-free plane graph on $n$ $(n\geq 6)$ vertices with $\delta(G)\geq 3$.
    \begin{enumerate}[label=(\roman*)]
        \item If $B$ is a nontrivial triangular-block (that is, not $B_2$), then none of the exterior faces can have length $5$. \label{it:faces:no5face}
        \item If $B$ is in $\{B_{5,a},B_{5,b},B_{5,c},B_{4,a}\}$, then none of the exterior faces can have length $4$. \label{it:faces:no4face}
        \item If $B$ is in $\{B_{5,d},B_{4,b}\}$ and an exterior face of $B$ has length $4$, then that $4$-face must share a $2$-path with $B$ (shown in blue in Figures~\ref{fig:B5d} and~\ref{fig:B4b}) and the other edges of the face must be in trivial triangular-blocks. \label{it:faces:yes4face}
        \item No two $4$-faces can be adjacent to each other. \label{it:faces:two4faces}
    \end{enumerate}
\end{proposition}

\begin{proof}
    \begin{itemize}
        \item[\ref{it:faces:no5face}] Observe that any pair of consecutive exterior vertices of a nontrivial triangular-block has a path of length $2$ (counted by the number of edges) between them and any pair of nonconsecutive exterior vertices has a path of length $3$ between them. So having a face of length $5$ incident to this triangular-block would yield a $C_6$, a contradiction.

        \item[\ref{it:faces:no4face}] If $B$ is in $\{B_{5,a},B_{5,b},B_{5,c},B_{4,a}\}$, then any pair of consecutive exterior vertices of the listed triangular-blocks has a path of length $3$ between them. It remains to consider nonconsecutive vertices for $\{B_{5,b},B_{5,c}\}$. For $B_{5,b}$ each pair of nonconsecutive exterior vertices has a path of length $3$ between them. In the case where $B$ is $B_{5,c}$, this is true for all pairs without an edge between them. As for the other pairs, if they are in the same $4$-face, then at least one of the degree-$2$ vertices in $B$ must have degree $2$ in $G$, a contradiction.

        \item[\ref{it:faces:yes4face}] In both $B_{5,d}$ and $B_{4,b}$, any pair of consecutive exterior vertices has a path of length $3$ between them. For $B_{5,d}$, in Figure~\ref{fig:B5d}, we see that there is a path of length $4$ between $v_2$ and $v_4$ and so the only way a $4$-face can be adjacent to $B$ is via a $2$-path with endvertices $v_1$ and $v_3$. In fact, because there is no vertex of degree $2$, the path must be $v_1v_4v_3$.  For $B_{4,b}$, in Figure~\ref{fig:B5d}, we see that because $B$ cannot have a vertex of degree $2$, the $4$-face and $B$ cannot share the path $v_2v_1v_4$ or the path $v_2v_3v_4$. Thus the only paths that can share a boundary with a $4$-face are $v_1v_4v_3$ and $v_1v_2v_3$.

        As to the other blocks that form edges of such a $4$-face. In Figure~\ref{fig:4face}, we see that if, say, $v_1u$ is in a nontrivial triangular-block, then there is a vertex $w$ in that block, in which case $wv_1xv_4v_3uw$ forms a $6$-cycle, a contradiction.

        \item[\ref{it:faces:two4faces}] If two $4$-faces share an edge, then there is a $6$-cycle formed by deleting that edge. If two $4$-faces share a $2$-path, then the midpoint of that path is a vertex of degree $2$ in $G$. In both cases, a contradiction.
    \end{itemize}
\end{proof}

\begin{figure}[ht]
    \centering
    \begin{tikzpicture}[scale=0.2]
        \coordinate (v1) at (0,12);
        \coordinate (v2) at (-6,6);
        \coordinate (v3) at (0,0);
        \coordinate (v4) at (8,6);
        \coordinate (v5) at (4,6);
        \coordinate (u) at (14,6);
        \coordinate (w) at (10,14);
        \draw[blue,very thick] (v1) -- (u) -- (v3);
        \draw[blue,dashed,very thick] (v1) -- (w) -- (u);
        \draw[very thick] (v1) -- (v4) -- (v3) -- (v2) -- (v1);
        \draw[very thick] (v1) -- (v3);
        \draw[very thick] (v1) -- (v5) -- (v3);
        \draw[very thick] (v4) -- (v5);
        \draw[fill=black] (v1) circle(\circ)  node[label=above:$v_1$] {};
        \draw[fill=black] (v3) circle(\circ)  node[label=below:$v_3$] {};
        \draw[fill=black] (v2) circle(\circ);
        \draw[fill=black] (v4) circle(\circ) node[label=right:$v_4$] {};
        \draw[fill=black] (v5) circle(\circ)  node[label=left:$x$] {};
        \draw[fill=black] (u) circle(\circ) node[label=right:$u$] {};
        \draw[fill=black] (w) circle(\circ) node[label=right:$w$] {};
    \end{tikzpicture}
    \qquad\qquad\qquad\qquad
    \begin{tikzpicture}[scale=0.2]
        \coordinate (v1) at (0,12);
        \coordinate (v2) at (-6,6);
        \coordinate (v3) at (0,0);
        \coordinate (v4) at (6,6);
        \coordinate (u) at (12,6);
        \coordinate (w) at (8,14);
        \draw[blue,very thick] (v3) -- (u) -- (v1);
        \draw[blue,dashed,very thick] (v1) -- (w) -- (u);
        \draw[very thick] (v3) -- (v4) -- (v1) -- (v2) -- (v3);
        \draw[very thick] (v4) -- (v2);
        \draw[fill=black] (v1) circle(\circ)  node[label=above:$v_1$] {};
        \draw[fill=black] (v2) circle(\circ)  node[label=left:$x$] {};
        \draw[fill=black] (v3) circle(\circ)  node[label=below:$v_3$] {};
        \draw[fill=black] (v4) circle(\circ)  node[label=right:$v_4$] {};
        \draw[fill=black] (u) circle(\circ)  node[label=right:$u$] {};
        \draw[fill=black] (w) circle(\circ)  node[label=right:$w$] {};
    \end{tikzpicture}
    \caption{Proposition~\ref{prop:faces}\ref{it:faces:yes4face}: The blocks defined by blue edges must be trivial.}
    \label{fig:4face}
\end{figure}
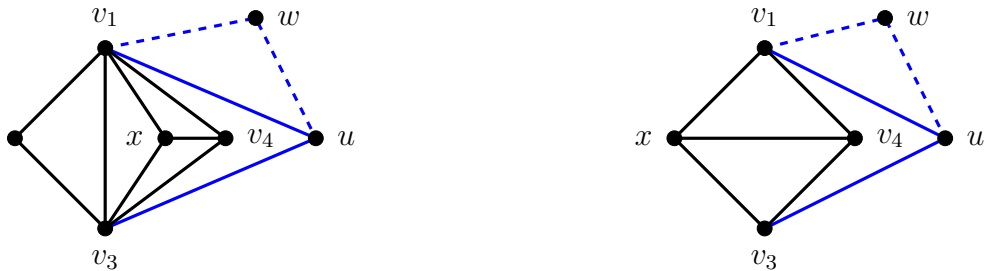

To show the existence of such a partition we need the following lemmas.
\begin{lemma}\label{lem:easyblock}
    Let $G$ be a $2$-connected, $C_6$-free plane graph on $n$ $(n\geq 6)$ vertices with $\delta(G)\geq 3$. If $B$ is a triangular-block in $G$ such that $B\notin \{B_{5,d}, B_{4,b}\}$, then $7 f(B) + 2 n(B) - 5 e(B) \leq 0$.
\end{lemma}

\begin{proof}
We separate the proof into several cases.

\subsubsection*{Case 1: $B$ is $B_{5,a}$.}
    Let $v_1$, $v_2$ and $v_3$ be the exterior vertices of $K^{-}_5$. At least two of them must be junction vertices, otherwise $G$ contains a cut vertex. We consider $2$ possibilities to justify this case.

    \begin{enumerate}[label=(\alph*)]
        \item Let $B$ be $B_{5,a}$ with $3$ junction vertices (see Figure \ref{fig:B5a}(a)). By Proposition~\ref{prop:faces}, every exterior edge in $B$ is contained in an exterior face with length at least $7$. Thus, $f(B) = (\#~\mbox{interior faces of $B$}) + \sum\limits_{e\in Ext(B)} f_F(e) \leq 5 + 3/7$. Moreover, every junction vertex is contained in at least $2$ triangular-blocks, so we have $n(B)\leq 2 + 3/2$. With $e(B)=9$, we obtain $7f(B)+2n(B)-5e(B)\leq 0$. \label{case1a}

        \item Let $B$ be $B_{5,a}$ with $2$ junction vertices, say $v_2$ and $v_3$ (see Figure \ref{fig:B5a}(b)). Let $F$ and $F_1$ are exterior faces of the exterior edge $v_2v_3$ and exterior path $v_2v_1v_3$ of the triangular-block respectively. Notice that $v_1v_2$ and $v_2v_3$ are the adjacent exterior edges in the same face $F_1$, hence $|C(F_1)|\geq 8$. By Definition \ref{bvvc}, we have $f_{F_1}(v_1v_2)+f_{F_1}(v_1v_3) \leq 1/7$. Because there can be no $C_6$, one can see that regardless of the configuration of the $B_{K_5^-}$, it is the case that $f_F(v_2v_3) \leq 1/7$. Thus, $f(B)\leq 5 + 2/7$. Moreover, since $v_1$ and $v_3$ are contained in at least $2$ triangular-blocks, we have $n(B) \leq 3 + 2/2$. With $e(B)=9$, we obtain $7f(B)+2n(B)-5e(B)\leq 0$. \label{case1b}

        \begin{figure}[ht]
            \centering
            \begin{tikzpicture}[scale=0.2]
                \clip (-7,-11) rectangle (7,13); 
                \coordinate (v1) at (0,8);
                \coordinate (v2) at (-4,0);
                \coordinate (v3) at (4,0);
                \draw[red,very thick,dashed] (v2) to[out=180,in=120,distance=200pt] (v1);
                \draw[red,very thick,dashed] (v3) to[out=0,in=60,distance=200pt] (v1);
                \draw[red,very thick,dashed] (v2) to[out=225,in=315,distance=400pt] (v3);
                \draw[very thick] (v2) -- (v1) -- (v3) -- (v2);
                \draw[fill=black] (v1) circle(\circ)  node[label=above:$v_1$] {};
                \draw[fill=black] (v2) circle(\circ)  node[label=below:$v_2$] {};
                \draw[fill=black] (v3) circle(\circ)  node[label=below:$v_3$] {};
                \node at (0,3) {$K_5^-$};
                \node at (0,-10){$(a)$};
            \end{tikzpicture}
            \qquad\qquad\qquad\qquad
            \begin{tikzpicture}[scale=0.2]
                \clip (-7, -11) rectangle (7, 13); 
                \coordinate (v1) at (0,8);
                \coordinate (v2) at (-4,0);
                \coordinate (v3) at (4,0);
                \draw[red,very thick,dashed] (v2) to[out=120,in=60,distance=550pt] (v3);
                \draw[red,very thick,dashed] (v2) to[out=225,in=315,distance=400pt] (v3);
                \draw[very thick] (v2) -- (v1) -- (v3) -- (v2);
                \draw[fill=black] (v1) circle(\circ)  node[label=above:$v_1$] {};
                \draw[fill=black] (v2) circle(\circ)  node[label=below:$v_2$] {};
                \draw[fill=black] (v3) circle(\circ)  node[label=below:$v_3$] {};
                \node at (4,8) {$F_1$};
                \node at (0,-3) {$F$};
                \node at (0,3) {$K_5^-$};
                \node at (0,-10){$(b)$};
            \end{tikzpicture}
            \caption{A $B_{5,a}$ triangular-block with $3$ and $2$ junction vertices, respectively.}
            \label{fig:B5a}
        \end{figure}
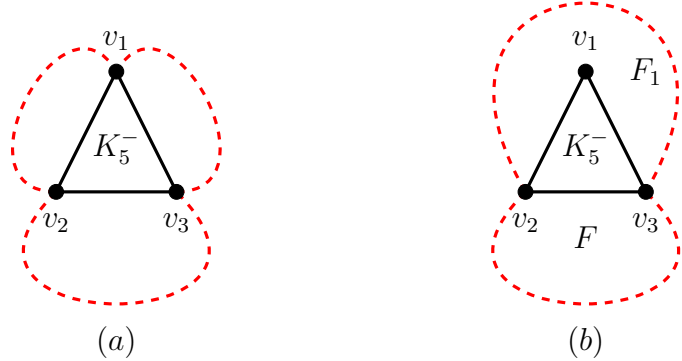
    \end{enumerate}

\subsubsection*{Case $2$: $B$ is in $\{B_{4,a},B_{5,b},B_{5,c}\}$.}

    \begin{enumerate}[label=(\alph*)]
        \item Let $B$ be a $B_{4,a}$. By Proposition~\ref{prop:faces}, each face incident to this triangular-block has length at least $7$. So, $f(B) \leq 3 + 3/7$. Because there is no cut-vertex, this triangular-block must have at least two junction vertices, hence $n(B) \leq 2 + 2/2$. With $e(B) = 6$, we obtain $7 f(B) + 2 n(B) - 5 e(B) \leq 0$. \label{case2a}

        \item Let $B$ be a $B_{5,b}$. There are $4$ faces inside the triangular-block and each face incident to this triangular-block has length at least $7$. So, $f(B) \leq 4 + 4/7$. Because there is no cut-vertex, this triangular-block must have at least two junction vertices, hence $n(B) \leq 3 + 2/2$. With $e(B) = 8$, we obtain $7 f(B) + 2 n(B) - 5 e(B) \leq 0$, as seen in Table~\ref{tab:5block}. \label{case2b}

        \item Let $B$ be a $B_{5,c}$. Similarly, $f(B) \leq 3 + 5/7$ and because there are at least two junction vertices, $n(B) \leq 3 + 2/2$. With $e(B) = 7$, we obtain $7 f(B) + 2 n(B) - 5 e(B) \leq -1$. \label{case2c}
    \end{enumerate}

\subsubsection*{Case $3$: $B$ is $B_{3}$.}
    Let $v_1$, $v_2$ and $v_3$ be the exterior vertices of triangular-block $B$. Each of these three must be junction vertices since there is no degree $2$ vertex in $G$, which implies that each is contained in at least $2$ triangular-blocks. We consider two possibilities:

    \begin{enumerate}[label=(\alph*)]
        \item  Let the three exterior vertices be contained in exactly $2$ triangular-blocks. By Proposition~\ref{prop:faces}\ref{it:faces:no5face}, the length of each exterior face is either $4$ or at least $7$. We want to show that at most one exterior face has length $4$.

        If not, then let $x_1$ be a vertex that is in two such faces. Consider the triangular-block incident to $B$ at $x_1$, call it $B'$. By Proposition~\ref{prop:faces}, $B'$ is not in $\{B_{5,a},B_{5,b},B_{5,c},B_{4,a}\}$.

        If $B'$ is in $\{B_{5,d},B_{4,b},B_3\}$, then the triangular-block has vertices $\ell_2,\ell_3$, each adjacent to $x_1$ and the length-$4$ faces consist of $\{v_1,\ell_2,m_2,v_2\}$ and $\{v_1,\ell_3,m_3,v_3\}$. Either $\ell_2\sim\ell_3$ (in which case $\ell_2m_2v_2v_3m_3\ell_3\ell_2$ is a $6$-cycle, see Figure~\ref{fig:B3}(a)) or there is a $\ell'$ distinct from $v_1$ that is adjacent to both $\ell_2$ and $\ell_3$ (in which case $\ell'\ell_2m_2v_2v_1\ell_3\ell_2$ is a $6$-cycle, see Figure~\ref{fig:B3}(b)).

        If $B'$ is $B_2$, then the trivial triangular-block is $\{v_1,\ell\}$, in which case $\{\ell,m_2,v_2,v_1,v_3,m_3\}$ is a $C_6$, see Figure~\ref{fig:B3}(c). Thus, we may conclude that if each of the three exterior vertices are in exactly $2$ triangular-blocks, then $f(B) \leq 1 + 2/7 + 1/4$ and $n(B) \leq 3/2$. With $e(B) = 3$, we obtain $7 f(B) + 2 n(B) - 5 e(B) \leq -5/4$.\label{case3a}

        \begin{figure}[ht]
            \centering
            \begin{tikzpicture}[scale=0.2]
                \coordinate (v1) at (0,8);
                \coordinate (v2) at (-4,0);
                \coordinate (v3) at (4,0);
                \coordinate (l2) at (-4,12);
                \coordinate (l3) at (4,12);
                \coordinate (m2) at (-8,4);
                \coordinate (m3) at (8,4);
                \draw[very thick] (v2) -- (v1) -- (v3) -- (v2);
                \draw[red,very thick,fill=red!15] (v1) -- (l2) -- (l3) -- (v1);
                \draw[dashed,very thick] (v2) -- (m2) -- (l2);
                \draw[dashed,very thick] (v3) -- (m3) -- (l3);
                \draw[red,very thick] (l2) -- (l3);
                \draw[fill=black] (v1) circle(\circ)  node[label=right:$v_1$] {};
                \draw[fill=black] (v2) circle(\circ)  node[label=below:$v_2$] {};
                \draw[fill=black] (v3) circle(\circ)  node[label=below:$v_3$] {};
                \draw[fill=red] (l2) circle(\circ)  node[label=left:$\ell_2$] {};
                \draw[fill=red] (l3) circle(\circ)  node[label=right:$\ell_3$] {};
                \draw[fill=black] (m2) circle(\circ)  node[label=right:$m_2$] {};
                \draw[fill=black] (m3) circle(\circ)  node[label=left:$m_3$] {};
                \node at (0,3) {$B$};
                \node[red] at (0,10.5) {$B'$};
                \node at (0,16) {};
                \node at (0,-6) {$(a)$};
            \end{tikzpicture}
            \qquad\qquad
            \begin{tikzpicture}[scale=0.2]
                \coordinate (v1) at (0,8);
                \coordinate (v2) at (-4,0);
                \coordinate (v3) at (4,0);
                \coordinate (l2) at (-4,12);
                \coordinate (l3) at (4,12);
                \coordinate (lp) at (0,16);
                \coordinate (m2) at (-8,4);
                \coordinate (m3) at (8,4);
                \draw[very thick] (v2) -- (v1) -- (v3) -- (v2);
                \draw[red,very thick,fill=red!15] (v1) -- (l2) -- (lp) -- (l3) -- (v1);
                \draw[dashed,very thick] (v2) -- (m2) -- (l2);
                \draw[dashed,very thick] (v3) -- (m3) -- (l3);
                \draw[fill=black] (v1) circle(\circ)  node[label=right:$v_1$] {};
                \draw[fill=black] (v2) circle(\circ)  node[label=below:$v_2$] {};
                \draw[fill=black] (v3) circle(\circ)  node[label=below:$v_3$] {};
                \draw[fill=red] (l2) circle(\circ)  node[label=left:$\ell_2$] {};
                \draw[fill=red] (l3) circle(\circ)  node[label=right:$\ell_3$] {};
                \draw[fill=red] (lp) circle(\circ)  node[label=right:$\ell'$] {};
                \draw[fill=black] (m2) circle(\circ)  node[label=right:$m_2$] {};
                \draw[fill=black] (m3) circle(\circ)  node[label=left:$m_3$] {};
                \node at (0,3) {$B$};
                \node[red] at (0,12) {$B'$};
                \node at (0,16) {};
                \node at (0,-6) {$(a)$};
            \end{tikzpicture}
            \qquad\qquad
            \begin{tikzpicture}[scale=0.2]
                \coordinate (v1) at (0,8);
                \coordinate (v2) at (-4,0);
                \coordinate (v3) at (4,0);
                \coordinate (l) at (0,15);
                \coordinate (m2) at (-8,6);
                \coordinate (m3) at (8,6);
                \draw[very thick] (v2) -- (v1) -- (v3) -- (v2);
                \draw[red,very thick] (v1) -- (l);
                \draw[dashed,very thick] (v2) -- (m2) -- (l);
                \draw[dashed,very thick] (v3) -- (m3) -- (l);
                \draw[fill=black] (v1) circle(\circ)  node[label=right:$v_1$] {};
                \draw[fill=black] (v2) circle(\circ)  node[label=below:$v_2$] {};
                \draw[fill=black] (v3) circle(\circ)  node[label=below:$v_3$] {};
                \draw[fill=red] (l) circle(\circ)  node[label=right:$\ell$] {};
                \draw[fill=black] (m2) circle(\circ)  node[label=right:$m_2$] {};
                \draw[fill=black] (m3) circle(\circ)  node[label=left:$m_3$] {};
                \node at (0,3) {$B$};
                \node[red] at (-1.5,10.5) {$B'$};
                \node at (0,16) {};
                \node at (0,-6) {$(a)$};
            \end{tikzpicture}
            \caption{A $B_3$ triangular-block, $B$ and the various cases of what must occur if $B$ is incident to two $4$-faces.}
            \label{fig:B3}
        \end{figure}
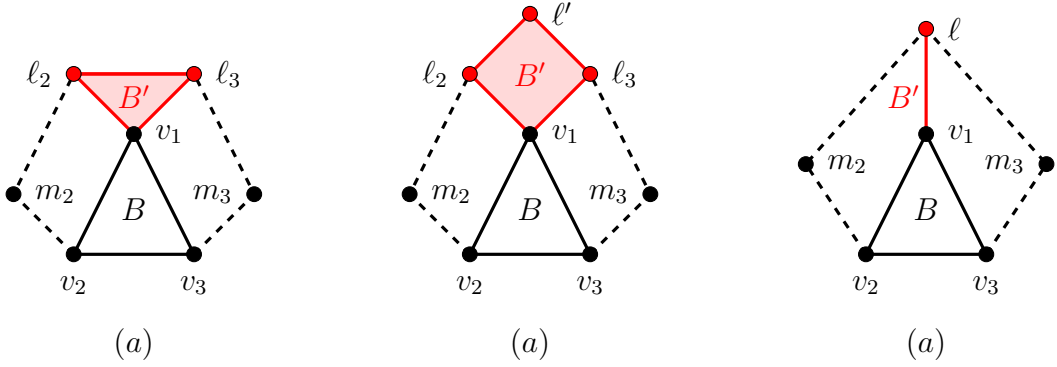

        \item  Let at least one exterior vertex be contained in at least $3$ triangular-blocks and the others be contained at least $2$ triangular-blocks. In this case, we have $f(B) \leq 1 + 3/4$ and $n(B) \leq 2/2 + 2/3$. With $e(B) = 3$, we obtain $7 f(B) + 2 n(B) - 5 e(B) \leq -1/12$.\label{case3b}
    \end{enumerate}

\subsubsection*{Case $4$: $B$ is $B_{2}$.}
    Note that the fact that there is no vertex of degree $2$ gives that if an endvertex is in exactly two triangular-blocks, then the other one cannot be a $B_2$. We consider three possibilities:

    \begin{enumerate}[label=(\alph*)]
        \item  Let each endvertex be contained in exactly $2$ triangular-blocks. Since neither of the triangular-blocks incident to $B$ can be trivial, they cannot be incident to a face of length $5$ by Proposition~\ref{prop:faces}\ref{it:faces:no5face}. Thus, $B$ cannot be incident to a face of length $5$. Moreover, the two faces incident to $B$ cannot both be of length $4$, again by Proposition~\ref{prop:faces}\ref{it:faces:two4faces}. Hence, $f(B) \leq 1/4 + 1/7$. Clearly $n(B) \leq 2/2$ and with $e(B) = 1$, we obtain $7 f(B) + 2 n(B) - 5 e(B) \leq -1/4$.\label{case4a}

        \item  Let one endvertex be contained in exactly $2$ triangular-blocks and the other endvertex be contained in at least $3$ triangular-blocks. This is similar to case \ref{case4a} in that neither face can have length $5$ and they cannot both have length $4$. The only difference is that $n(B) \leq 1/2 + 1/3$ and so $7 f(B) + 2 n(B) - 5 e(B) \leq -7/12$.\label{case4b}

        \item  Let each endvertex be contained in at least $3$ triangular-blocks. The two faces cannot both be of length $4$ by Proposition~\ref{prop:faces}\ref{it:faces:two4faces}. Hence, $f(B) \leq 1/4 + 1/5$ and $n(B) \leq 2/3$. With $e(B) = 1$, we obtain $7 f(B) + 2 n(B) - 5 e(B) \leq -31/60$.\label{case4c}
    \end{enumerate}
\end{proof}

\begin{lemma}\label{lem:B5d}
    Let $G$ be a $2$-connected, $C_6$-free plane graph on $n$ $(n\geq 6)$ vertices with $\delta(G)\geq 3$. If $B$ is $B_{5,d}$, then $7 f(B) + 2 n(B) - 5 e(B) \leq 1/2$. Moreover, $7 f(B) + 2 n(B) - 5 e(B) \leq 0$ unless $B$ shares a $2$-path with a $4$-face.
\end{lemma}

\begin{figure}[ht]
    \centering
    \begin{tikzpicture}[scale=0.2]
        \coordinate (v1) at (0,12);
        \coordinate (v4) at (8,6);
        \coordinate (v5) at (4,6);
        \coordinate (v3) at (0,0);
        \coordinate (v2) at (-6,6);
        \draw[very thick] (v1) -- (v4) -- (v3) -- (v2) -- (v1);
        \draw[very thick] (v1) -- (v3);
        \draw[very thick] (v1) -- (v5) -- (v3);
        \draw[very thick] (v4) -- (v5);
        \draw[fill=black] (v1) circle(\circ)  node[label=above:$v_1$] {};
        \draw[fill=black] (v4) circle(\circ) node[label=right:$v_4$] {};
        \draw[fill=black] (v5) circle(\circ)  node[label=left:$v_5$] {};
        \draw[fill=black] (v3) circle(\circ)  node[label=below:$v_3$] {};
        \draw[fill=black] (v2) circle(\circ)  node[label=left:$v_2$] {};
        \node at (0,-6) {$(a)$};
    \end{tikzpicture}
    \qquad\qquad\qquad\qquad
    \begin{tikzpicture}[scale=0.2]
        \coordinate (v1) at (0,12);
        \coordinate (v4) at (8,6);
        \coordinate (v5) at (4,6);
        \coordinate (v3) at (0,0);
        \coordinate (v2) at (-6,6);
        \coordinate (u) at (14,6);
        \draw[blue,very thick] (v1) -- (u) -- (v3);
        \draw[very thick] (v1) -- (v4) -- (v3) -- (v2) -- (v1);
        \draw[very thick] (v1) -- (v3);
        \draw[very thick] (v1) -- (v5) -- (v3);
        \draw[very thick] (v4) -- (v5);
        \draw[fill=black] (v1) circle(\circ)  node[label=above:$v_1$] {};
        \draw[fill=black] (v4) circle(\circ) node[label=right:$v_4$] {};
        \draw[fill=black] (v5) circle(\circ)  node[label=left:$v_5$] {};
        \draw[fill=black] (v3) circle(\circ)  node[label=below:$v_3$] {};
        \draw[fill=black] (v2) circle(\circ)  node[label=left:$v_2$] {};
        \draw[fill=black] (u) circle(\circ);
        \node at (0,-6) {$(b)$};
    \end{tikzpicture}
    \caption{A $B_{5,d}$ triangular-block and how a $4$-face must be incident to it.}
    \label{fig:B5d}
\end{figure}
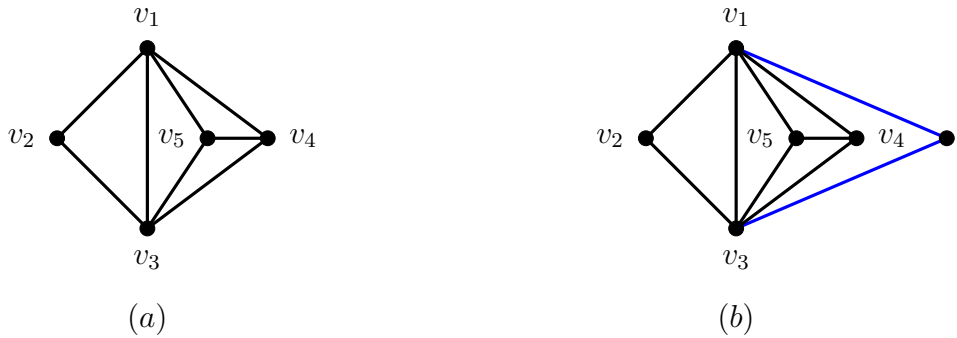

\begin{proof}
Let $B$ be $B_{5,d}$ with vertices $v_1$, $v_2$, $v_3$, $v_4$, and $v_5$, as shown in Figure~\ref{fig:B5d}(a). By Proposition~\ref{prop:faces}\ref{it:faces:no5face}, no exterior face of $B$ can have length $5$. By Proposition~\ref{prop:faces}\ref{it:faces:yes4face}, if there is an exterior face of $B$ that has length $4$, this $4$-face must contain the path $v_1v_4v_3$.

Moreover, since there is no vertex of degree $2$, $v_2$ is a junction vertex. Because $G$ has no cut-vertex, there is at least one other junction vertex. We may consider the following cases:

\begin{enumerate}[label=(\alph*)]
    \item Let $v_4$ be a junction vertex. This prevents an exterior face of length $4$. Thus, each exterior face has length at least $7$. Hence, $f(B) \leq 4 + 4/7$ and $n(B) \leq 3 + 2/2$. With $e(B) = 8$, we obtain $7 f(B) + 2 n(B) - 5 e(B) \leq 0$. \label{case5a}

    \item Let $v_4$ fail to be a junction vertex and exactly one of $v_1,v_3$ be a junction vertex. Without loss of generality let it be $v_3$. In this case, again, each exterior face has length\footnote{In fact, it can be shown that the length of the exterior face containing the path $v_2v_1v_4v_3$ is at least $9$. This yields $f(B) \leq 4 + 1/7 + 3/9$ and $7 f(B) + 2 n(B) - 5 e(B) \leq -2/3$. However, this precision is unnecessary.}  at least $7$. Again, $f(B) \leq 4 + 4/7$ and $n(B) \leq 3 + 2/2$. With $e(B) = 8$, we obtain $7 f(B) + 2 n(B) - 5 e(B) \leq 0$. \label{case5b}

    \item Let $v_4$ fail to be a junction vertex and both $v_1$ and $v_3$ be junction vertices. Here either the exterior path $v_1v_4v_3$ is part of an exterior face of length at least $4$ or each edge must be in a face of length at least $7$.
    If the exterior face is of length at least $7$, then $f(B) \leq 4 + 4/7$, otherwise $f(B) \leq 4 + 2/4 + 2/7$. In both cases, $n(B) \leq 2 + 3/2$ and $e(B) = 8$. Hence we obtain $7 f(B) + 2 n(B) - 5 e(B) \leq -1$ in the first instance and $7 f(B) + 2 n(B) - 5 e(B) \leq 1/2$ in the case where $B$ is incident to a $4$-face.\label{case5c}
\end{enumerate}
\end{proof}

\begin{lemma}\label{lem:B4b}
    Let $G$ be a $2$-connected, $C_6$-free plane graph on $n$ $(n\geq 6)$ vertices with $\delta(G)\geq 3$. If $B$ is $B_{4,b}$, then $7 f(B) + 2 n(B) - 5 e(B) \leq 4/3$. Moreover, $7 f(B) + 2 n(B) - 5 e(B) \leq 1/6$ if $B$ shares a $2$-path with exactly one $4$-face and $7 f(B) + 2 n(B) - 5 e(B) \leq 0$ if $B$ fails to share a $2$-path with any $4$-face.
\end{lemma}

\begin{figure}[ht]
    \centering
    \begin{tikzpicture}[scale=0.2]
        \coordinate (v1) at (0,12);
        \coordinate (v2) at (-6,6);
        \coordinate (v3) at (0,0);
        \coordinate (v4) at (6,6);
        \draw[very thick] (v3) -- (v4) -- (v1) -- (v2) -- (v3);
        \draw[very thick] (v4) -- (v2);
        \draw[fill=black] (v1) circle(\circ)  node[label=above:$v_1$] {};
        \draw[fill=black] (v2) circle(\circ)  node[label=left:$v_2$] {};
        \draw[fill=black] (v3) circle(\circ)  node[label=below:$v_3$] {};
        \draw[fill=black] (v4) circle(\circ)  node[label=right:$v_4$] {};
        \node at (0,-6) {$(a)$};
    \end{tikzpicture}
    \qquad
    \begin{tikzpicture}[scale=0.2]
        \coordinate (v1) at (0,12);
        \coordinate (v2) at (-6,6);
        \coordinate (v3) at (0,0);
        \coordinate (v4) at (6,6);
        \coordinate (u) at (12,6);
        \draw[blue,very thick] (v3) -- (u) -- (v1);
        \draw[very thick] (v3) -- (v4) -- (v1) -- (v2) -- (v3);
        \draw[very thick] (v4) -- (v2);
        \draw[fill=black] (v1) circle(\circ)  node[label=above:$v_1$] {};
        \draw[fill=black] (v2) circle(\circ)  node[label=left:$v_2$] {};
        \draw[fill=black] (v3) circle(\circ)  node[label=below:$v_3$] {};
        \draw[fill=black] (v4) circle(\circ)  node[label=right:$v_4$] {};
        \draw[fill=black] (u) circle(\circ);
        \node at (0,-6) {$(b)$};
    \end{tikzpicture}
    \qquad
    \begin{tikzpicture}[scale=0.2]
        \coordinate (v1) at (0,12);
        \coordinate (v2) at (-6,6);
        \coordinate (v3) at (0,0);
        \coordinate (v4) at (6,6);
        \coordinate (u1) at (12,6);
        \coordinate (u2) at (-12,6);
        \draw[blue,very thick] (v3) -- (u1) -- (v1);
        \draw[blue,very thick] (v3) -- (u2) -- (v1);
        \draw[very thick] (v3) -- (v4) -- (v1) -- (v2) -- (v3);
        \draw[very thick] (v4) -- (v2);
        \draw[fill=black] (v1) circle(\circ)  node[label=above:$v_1$] {};
        \draw[fill=black] (v2) circle(\circ)  node[label=left:$v_2$] {};
        \draw[fill=black] (v3) circle(\circ)  node[label=below:$v_3$] {};
        \draw[fill=black] (v4) circle(\circ)  node[label=right:$v_4$] {};
        \draw[fill=black] (u1) circle(\circ);
        \draw[fill=black] (u2) circle(\circ);
        \node at (0,-6) {$(c)$};
    \end{tikzpicture}
    \caption{A $B_{4,b}$ triangular-block and how a $4$-face must be incident to it.}
    \label{fig:B4b}
\end{figure}
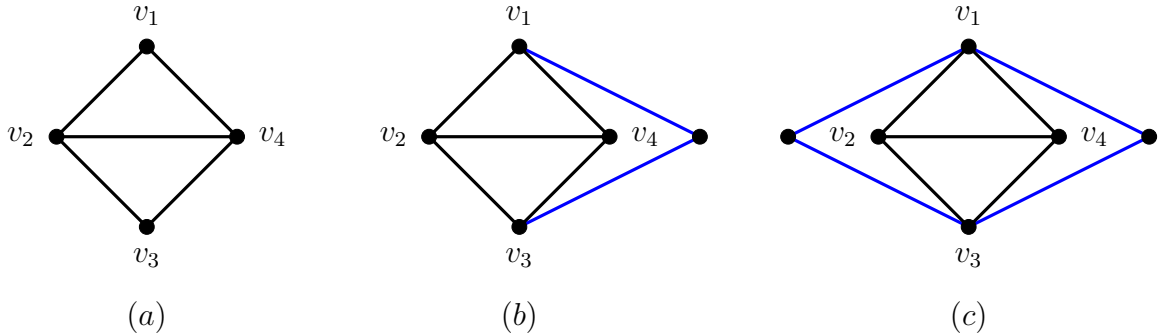

\begin{proof}
Let $B$ be with vertices $v_1$, $v_2$, $v_3$, and $v_4$, as shown in Figure~\ref{fig:B4b}(a). By Proposition~\ref{prop:faces}\ref{it:faces:no5face}, no exterior face of $B$ can have length $5$. If there is an exterior face of $B$ that has length $4$, it is easy to verify that being $C_6$-free and having no vertex of degree $2$ means that the junction vertices must be $v_1$ and $v_3$. We may consider the following cases.

\begin{enumerate}[label=(\alph*)]
    \item Let either $v_2$ or $v_4$ be a junction vertex and, without loss of generality, let it be $v_2$. All the exterior faces have length at least $7$ except for the possibility that the path $v_1v_4v_3$ may form two sides of a $4$-face. Hence, $f(B) \leq 2 + 2/4 + 2/7$ and $n(B) \leq 1 + 3/2$. With $e(B) = 5$, we obtain $7 f(B) + 2 n(B) - 5 e(B) \leq -1/2$. \label{case6a}

    \item Let neither $v_2$ nor $v_4$ be a junction vertex. Because there is no cut-vertex, this requires both $v_1$ and $v_3$ to be junction vertices. Hence, there are two exterior faces: One that shares the exterior path $v_1v_4v_3$ and the other shares the exterior path $v_1v_2v_3$. Each exterior face has length either $4$ or at least $7$. We consider several subcases: \label{case6b}

    \begin{enumerate}[label=(\roman*)]
        \item If both faces are of length at least $7$, then $f(B) \leq 2 + 4/7$, and $n(B) \leq 2 + 2/2$. With $e(B) = 5$, we obtain $7 f(B) + 2 n(B) - 5 e(B)\leq -1$. \label{case6bi}

        \item If only one of the exterior faces is of length $4$, then $f(B) \leq 2 + 2/7 + 2/4$. Moreover, at least one of $v_1$, $v_3$ must be a junction vertex for more than two triangular-blocks, otherwise either $v(G) = 5$ or the vertex incident to two blue edges in Figure~\ref{fig:B4b}(b) is a cut-vertex. Hence, $n(B) \leq 2 + 1/3 + 1/2$ and with $e(B) = 5$, we have $7 f(B) + 2 n(B) - 5 e(B) \leq 1/6$. \label{case6bii}

       \item Both exterior faces are of length $4$. Thus $f(B) \leq 2 + 4/4$. By Proposition~\ref{prop:faces}\ref{it:faces:yes4face}, the blocks represented by the blue edges in Figure~\ref{fig:B4b}(c) are each trivial. Hence $n(B) \leq 2 + 2/3$. With $e(B)=5$, we get $7 f(B) + 2 n(B) - 5 e(B) \leq 4/3$. \label{case6biii}
    \end{enumerate}
\end{enumerate}
\end{proof}

Tables~\ref{tab:5block} and~\ref{tab:432block} in Appendix~\ref{appendix} give a summary of Lemmas~\ref{lem:easyblock}, \ref{lem:B5d}, and \ref{lem:B4b}.

\begin{lemma}\label{lem:partition}
    Let $G$ be a $2$-connected, $C_6$-free plane graph on $n$ $(n\geq 6)$ vertices with $\delta(G)\geq 3$. Then the triangular-blocks of $G$ can be partitioned into sets, $\mathcal{P}_1$, $\mathcal{P}_2$,\dots, $\mathcal{P}_m$ such that $7\sum\limits_{B\in\mathcal{P}_i}f(B)+2\sum\limits_{B\in\mathcal{P}_i}n(B)-5\sum\limits_{B\in\mathcal{P}_i}e(B)\leq 0$ for all $i\in[m]$.
\end{lemma}

\begin{proof}
As it can be seen from Tables~\ref{tab:5block} and~\ref{tab:432block} in Appendix~\ref{appendix}, there are three possible cases where $7 f(B) + 2 n(B) - 5e(B)$ assumes a positive value.  We deal with each of these blocks as follows.
\begin{enumerate}[label=(\arabic*)]
    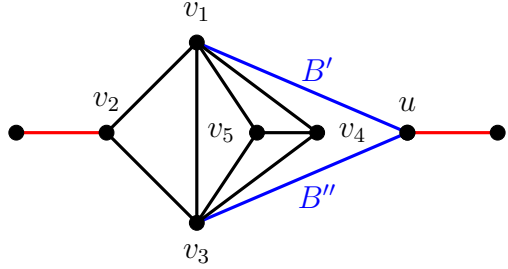
\begin{figure}[ht]
        \centering
        \begin{tikzpicture}[scale=0.2]
            \coordinate (v1) at (0,12);
            \coordinate (v2) at (-6,6);
            \coordinate (v3) at (0,0);
            \coordinate (v4) at (8,6);
            \coordinate (v5) at (4,6);
            \coordinate (u) at (14,6);
            \coordinate (leafl) at (-12,6);
            \coordinate (leafr) at (20,6);
            \coordinate (Bp) at (8,8);
            \coordinate (Bpp) at (8,4);
            \draw[blue,very thick] (v1) -- (u) -- (v3);
            \draw[red,very thick] (leafl) -- (v2);
            \draw[red,very thick] (leafr) -- (u);
            \draw[very thick] (v1) -- (v4) -- (v3) -- (v2) -- (v1);
            \draw[very thick] (v1) -- (v3);
            \draw[very thick] (v1) -- (v5) -- (v3);
            \draw[very thick] (v4) -- (v5);
            \draw[fill=black] (v1) circle(\circ) node[label=above:$v_1$] {};
            \draw[fill=black] (v2) circle(\circ)  node[label=above:$v_2$] {};
            \draw[fill=black] (v3) circle(\circ)  node[label=below:$v_3$] {};
            \draw[fill=black] (v4) circle(\circ) node[label=right:$v_4$] {};
            \draw[fill=black] (v5) circle(\circ)  node[label=left:$v_5$] {};
            \draw[fill=black] (u) circle(\circ) node[label=above:$u$] {};
            \draw[fill=black] (leafl) circle(\circ);
            \draw[fill=black] (leafr) circle(\circ);
            \draw[fill=black] (u) circle(\circ);
            \draw[blue] (Bp) node[label=above:$B'$] {};
            \draw[blue] (Bpp) node[label=below:$B''$] {};
        \end{tikzpicture}
        \caption{Structure of a $B_{5,d}$ if it is incident to a $4$-face, as in Lemma~\ref{lem:partition}. The triangular-blocks $B'$ and $B''$ are trivial.}
        \label{fig:B5dPartition}
    \end{figure}

    \item Let $B$ be a $B_{5,d}$ triangular-block as described in the proof of Lemma \ref{lem:B5d}\ref{case5c}. See Figure~\ref{fig:B5dPartition}.

    By Proposition~\ref{prop:faces}\ref{it:faces:yes4face}, the edges $v_1u$ and $v_3u$ are trivial triangular-blocks. Denote these triangular-blocks as $B'$ and $B''$. Consider $B'$. One of the exterior faces of $B'$ has length $4$ whereas by Proposition~\ref{prop:faces}\ref{it:faces:two4faces},the other has length at least $5$. It must have length at least $7$ because if it had length $5$, then the path $v_1v_3u$ would complete it to a $6$-cycle. Thus, $f(B') \leq 1/4 + 1/7$. Since the vertex $u$ cannot be of degree $2$, then this vertex is shared in at least three triangular-blocks. Thus, $n(B') \leq 1/2 + 1/3$. With $e(B') = 1$, we obtain $7 f(B') + 2 n(B') - 5 e(B') \leq -7/12$ and similarly, $7 f(B'') + 2 n(B'') - 5 e(B'') \leq -7/12$. Define $\mathcal{P}'=\{B,B',B''\}$. Thus, $7 \sum\limits_{B^*\in \mathcal{P}'} f(B^*) + 2 \sum\limits_{B^*\in \mathcal{P}'} n(B^*) - 5 \sum\limits_{B^*\in \mathcal{P}'} e(B^*) \leq 1/2 + 2(-7/12) = - 2/3$.

    Therefore, for each triangular-block in $G$ as described in Lemma \ref{lem:B5d}\ref{case5c}, it belongs to a set $\mathcal{P'}$ of three triangular-blocks such that $7 \sum\limits_{B^*\in \mathcal{P}'} f(B^*) + 2 \sum\limits_{B^*\in \mathcal{P}'} n(B^*) - 5 \sum\limits_{B^*\in \mathcal{P}'} e(B^*) \leq 0$.
    Denote such sets as $\mathcal{P}_1,\mathcal{P}_2,\dots,\mathcal{P}_{m_1}$ if they exist.

    \begin{figure}[ht]
        \centering
        \begin{tikzpicture}[scale=0.2]
            \coordinate (v1) at (0,12);
            \coordinate (v2) at (-6,6);
            \coordinate (v3) at (0,0);
            \coordinate (v4) at (6,6);
            \coordinate (u1) at (12,6);
            \coordinate (w) at (18,6);
            \coordinate (Bp) at (8,8);
            \coordinate (Bpp) at (8,4);
            \draw[red,very thick] (w) -- (u1);
            \draw[blue,very thick] (v3) -- (u1) -- (v1);
            \draw[very thick] (v3) -- (v4) -- (v1) -- (v2) -- (v3);
            \draw[very thick] (v4) -- (v2);
            \draw[fill=black] (v1) circle(\circ)  node[label=above:$v_1$] {};
            \draw[fill=black] (v2) circle(\circ)  node[label=left:$v_2$] {};
            \draw[fill=black] (v3) circle(\circ)  node[label=below:$v_3$] {};
            \draw[fill=black] (v4) circle(\circ)  node[label=right:$v_4$] {};
            \draw[fill=black] (u1) circle(\circ) node[label=above:$u_1$] {};
            \draw[fill=black] (w) circle(\circ);
            \draw[blue] (Bp) node[label=above:$B'$] {};
            \draw[blue] (Bpp) node[label=below:$B''$] {};
            \node at (0,-6) {$(a)$};
        \end{tikzpicture}
        \qquad\qquad\qquad
        \begin{tikzpicture}[scale=0.2]
            \coordinate (v1) at (0,12);
            \coordinate (v2) at (-6,6);
            \coordinate (v3) at (0,0);
            \coordinate (v4) at (6,6);
            \coordinate (u1) at (12,6);
            \coordinate (u2) at (-12,6);
            \coordinate (Bp) at (8,8);
            \coordinate (Bpp) at (8,4);
            \coordinate (Bppp) at (-8,8);
            \coordinate (Bpppp) at (-8,4);
            \draw[blue,very thick] (v3) -- (u1) -- (v1);
            \draw[blue,very thick] (v3) -- (u2) -- (v1);
            \draw[very thick] (v3) -- (v4) -- (v1) -- (v2) -- (v3);
            \draw[very thick] (v4) -- (v2);
            \draw[fill=black] (v1) circle(\circ)  node[label=above:$v_1$] {};
            \draw[fill=black] (v2) circle(\circ)  node[label=left:$v_2$] {};
            \draw[fill=black] (v3) circle(\circ)  node[label=below:$v_3$] {};
            \draw[fill=black] (v4) circle(\circ)  node[label=right:$v_4$] {};
            \draw[fill=black] (u1) circle(\circ) node[label=above:$u_1$] {};
            \draw[fill=black] (u2) circle(\circ) node[label=above:$u_2$] {};
            \draw[blue] (Bp) node[label=above:$B'$] {};
            \draw[blue] (Bpp) node[label=below:$B''$] {};
            \draw[blue] (Bppp) node[label=above:$B'''$] {};
            \draw[blue] (Bpppp) node[label=below:$B''''$] {};
            \node at (0,-6) {$(b)$};
        \end{tikzpicture}
        \caption{Structure of a $B_{4,b}$ triangular-block if it is incident to a $4$-face, as in Lemma~\ref{lem:partition}. The triangular-blocks $B'$, $B''$, $B'''$, and $B''''$ are all trivial.}
        \label{fig:B4bPartition}
    \end{figure}
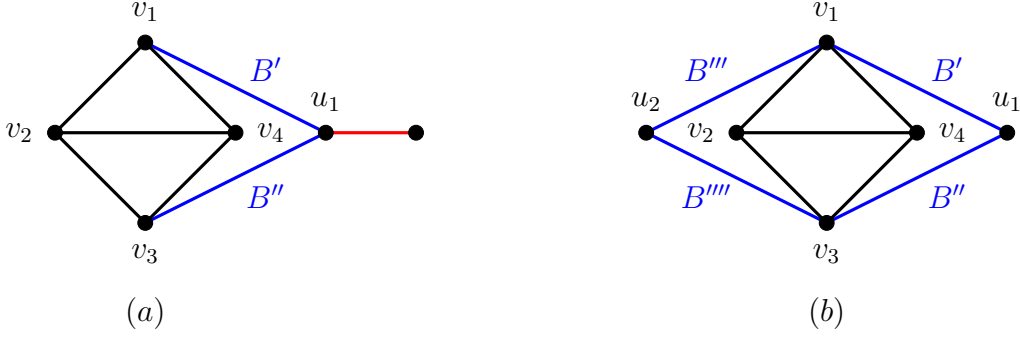

    \item Let $B$ be a $B_{4,b}$ triangular-block as described in the proof of Lemma~\ref{lem:B4b}\ref{case6b}\ref{case6bii}. See Figure~\ref{fig:B4bPartition}(a).

    By Proposition~\ref{prop:faces}\ref{it:faces:yes4face}, the edges $v_1u_1$ and $v_3u_1$ are trivial triangular-blocks. Denote them as $B'$ and $B''$, respectively. Consider $B'$. One of the exterior faces of $B'$ has length $4$ and by Proposition~\ref{prop:faces}\ref{it:faces:two4faces}, the other has length at least $5$. Thus, $f(B') \leq 1/4 + 1/5$. Since the vertex $u_1$ cannot be of degree $2$, then this vertex is shared in at least three triangular-blocks. Thus, $n(B') \leq 1/2 + 1/3$. With $e(B') = 1$, we obtain $7 f(B') + 2 n(B') - 5 e(B') \leq -11/60$ and similarly, $7 f(B'') + 2 n(B'') - 5 e(B'') \leq -11/60$. Define $\mathcal{P}''=\{B,B',B''\}$. Thus, $7 \sum\limits_{B^*\in \mathcal{P}''} f(B^*) + 2 \sum\limits_{B^*\in \mathcal{P}''} n(B^*) - 5 \sum\limits_{B^*\in \mathcal{P}''} e(B^*) \leq 1/6 + 2(-11/60) = -1/5$.

    Therefore, for each triangular-block in $G$ as described in Lemma \ref{lem:B4b}\ref{case6b}\ref{case6bii}, it belongs to a set $\mathcal{P''}$ of three triangular-blocks such that $7 \sum\limits_{B^*\in \mathcal{P}''} f(B^*) + 2 \sum\limits_{B^*\in \mathcal{P}''} n(B^*) - 5 \sum\limits_{B^*\in \mathcal{P}''} e(B^*) \leq 0$.
    Denote such sets as $\mathcal{P}_{m_1+1},\mathcal{P}_{m_1+2},\dots,\mathcal{P}_{m_2}$ if they exist.

    \item Let $B$ be a $B_{4,b}$ triangular-block as described in the proof of Lemma~\ref{lem:B4b}\ref{case6b}\ref{case6biii}. See Figure~\ref{fig:B4bPartition}(b).

    By Proposition~\ref{prop:faces}\ref{it:faces:yes4face}, the edges $v_1u_1$, $v_3u_1$, $v_1u_2$, and $v_3u_2$ are trivial triangular-blocks. Denote them as $B'$, $B''$, $B'''$ and $B''''$ respectively. Consider $B'$. One of the exterior faces of $B'$ has length $4$ whereas the other has length at least $5$. Thus, $f(B') \leq 1/4 + 1/5$. Since the vertex $u_1$ cannot be of degree $2$, then this vertex is shared in at least three triangular-blocks. Clearly $v_1$ is in at least three triangular-blocks. Thus, $n(B') \leq 2/3$. With $e(B') = 1$, we obtain $7 f(B') + 2 n(B') - 5 e(B') \leq -31/60$ and the same inequality holds for $B''$, $B'''$, and $B''''$.

    Define $\mathcal{P}'''=\{B,B',B'',B''',B''''\}$. Thus, $7 \sum\limits_{B^*\in \mathcal{P}''} f(B^*) + 2 \sum\limits_{B^*\in \mathcal{P}''} n(B^*) - 5 \sum\limits_{B^*\in \mathcal{P}''} e(B^*) \leq 4/3 + 4(-31/60) = -11/15$.

    Therefore, for each triangular-block in $G$ as described in Lemma \ref{lem:B4b}\ref{case6b}\ref{case6biii}, it belongs to a set $\mathcal{P'}$ of three triangular-blocks such that $7 \sum\limits_{B^*\in \mathcal{P}'''} f(B^*) + 2 \sum\limits_{B^*\in \mathcal{P}'''} n(B^*) - 5 \sum\limits_{B^*\in \mathcal{P}'''} e(B^*) \leq 0$.
    Denote such sets as $\mathcal{P}_{m_2+1},\mathcal{P}_{m_2+2},\dots,\mathcal{P}_{m_3}$ if they exist.
\end{enumerate}

Now define $\mathcal{P}_{m_3+1}=\mathcal{B}-\bigcup\limits_{i=1}^{m_3}\mathcal{P}_i$, where $\mathcal{B}$ is the set of all blocks of $G$. Clearly, for each block $B\in \mathcal{P}_{m_3+1}$, $7f(B)+2n(B)-5e(B)\leq 0$. Thus, $7\sum\limits_{B\in \mathcal{P}_{m_3+1}}f(B)+2\sum\limits_{B\in \mathcal{P}_{m_3+1}}n(B)-5\sum\limits_{B\in \mathcal{P}_{m_3+1}}e(B)\leq 0$. Putting $m:=m_3+1$ we got the partition $\mathcal{P}_1,\mathcal{P}_2,\dots,\mathcal{P}_{m}$ of $\mathcal{B}$ meeting the condition of the lemma.
\end{proof}
This completes the proof of Theorem~\ref{thm:main}.

\section{Proof of Theorem~\ref{thm:main_new}}\label{strongproof}

Let $G$ be a $C_6$-free plane graph. We will show that either $5 v(G) - 2 e(G) \geq 14$ or $v(G) \leq 17$.

If we delete a vertex $x$ from $G$, then
\begin{align*}
    5 v(G-x) - 2 e(G-x) &=      5 (v(G) - 1) - 2 (e(G) - \deg(x)) \\
                        &=      5 v(G) - 2 e(G) - 5 + 2 \deg(x) \\
                        &\geq   5 v(G) - 2 e(G) - 1.
\end{align*}

So, graph $G$ has an induced subgraph $G'$ with $\delta(G)\geq 3$ with
\begin{align}
    5 v(G) - 2 e(G) \geq 5 v(G') - 2 e(G') + \left(v(G)-v(G')\right) \label{eq:main_new:GtoGp}
\end{align}
In line with usual graph theoretic terminology, we call a maximal $2$-connected subgraph a \textbf{block}. Let $\mathcal{B}'$ denote the set of blocks of $G'$ with the $i^{\rm th}$ block having $n_i$ vertices and $e_i$ edges. Let $b$ be the total number of blocks of $G'$. Specifically, let $b_2$, $b_3$, $b_4$, and $b_5$ denote the number of blocks of size $2$, $3$, $4$, and $5$, respectively. Let $b_6$ denote the number of blocks of size at least $6$. Then we have $b = b_6 + b_5 + b_4 + b_3 + b_2$ and, using Table~\ref{tab:5n2e5}:
\begin{align}
    5 v(G') - 2 e(G')   &=      5 \left(\sum_{i=1}^b n_i - (b-1)\right) - 2 \sum_{i=1}^b e_i \nonumber \\
                        &=      \sum_{i=1}^b \left(5n_i - 2e_i - 5\right) + 5 \nonumber \\
                        &\geq   9 b_6 + 2 b_5 + 3 b_4 + 4 b_3 + 3 b_2 + 5 \label{eq:main_new:blocks}
\end{align}

\begin{table}
     \centering
     \begin{tabular}{|r|rl|l|}\hline
                    &   \multicolumn{3}{l|}{$\min$ of $5 n - 2 e - 5$} \\ \hline\hline
        $n\geq 6$   &   $14-5\; \geq$           & 9     & Theorem~\ref{thm:main} \\ \hline
        $n = 5$     &   $5(5)-2(9)-5\; \geq$    & 2     & $B_{5,a}$, Figure~\ref{fig:5blocks} \\ \hline
        $n = 4$     &   $5(4)-2(6)-5\; \geq$    & 3     & $B_{4,a}$, Figure~\ref{fig:432blocks} \\ \hline
        $n = 3$     &   $5(3)-2(3)-5\; \geq$    & 4     & $B_3$, Figure~\ref{fig:432blocks} \\ \hline
        $n = 2$     &   $5(2)-2(2)-5\; \geq$    & 3     & $B_2$, Figure~\ref{fig:432blocks} \\ \hline
     \end{tabular}
     \caption{Estimates of $5n-2e-5$ for various block sizes.}
     \label{tab:5n2e5}
\end{table}

Combining \eqref{eq:main_new:GtoGp} and \eqref{eq:main_new:blocks}, we obtain
\begin{align}
    5 v(G) - 2 e(G) \geq 9 b_6 + 2 b_5 + 3 b_4 + 4 b_3 + 3 b_2 + 5 + \left(v(G)-v(G')\right) \label{eq:main_new:Gbound}
\end{align}

If $b_6\geq 1$, then the right-hand side of \eqref{eq:main_new:Gbound} is at least $14$, as desired.

So, let us assume that $b_6=0$ and $b = b_5 + b_4 + b_3 + b_2$. Furthermore,
\begin{align}
    v(G')   &=  5 b_5 + 4 b_4 + 3 b_3 + 2 b_2 - (b-1) \nonumber \\
            &=  4 b_5 + 3 b_4 + 2 b_3 + b_2 + 1 . \label{eq:main_new:vGp}
\end{align}

So, substituting $2 b_5$ from \eqref{eq:main_new:vGp} into \eqref{eq:main_new:Gbound}, we have
\begin{align*}
    5 v(G) - 2 e(G)     &\geq   2 b_5 + 3 b_4 + 4 b_3 + 3 b_2 + 5 + \left(v(G)-v(G')\right) \\
                        &=      \left(\frac{1}{2}v(G') - \frac{3}{2} b_4 - b_3 - \frac{1}{2} b_2 - \frac{1}{2}\right) + 3 b_4 + 4 b_3 + 3 b_2 + 5 + \left(v(G)-v(G')\right) \\
                        &=      v(G) - \frac{1}{2}v(G') + \frac{3}{2} b_4 + 3 b_3 + \frac{5}{2} v_2 + \frac{9}{2} \\
                        &\geq   \frac{1}{2}v(G) + \frac{9}{2} ,
\end{align*}
which is strictly larger than $13$ if $v(G)\geq 18$. Since $5 v(G) - 2 e(G)$ is an integer, it is at least $14$ and this completes the proof of Theorem~\ref{thm:main_new}.

\begin{remark}
    Observe that for $n\geq 17$, the only graphs on $n$ vertices with $e$ edges such that $e > (5/2)n - 7$ have blocks of order $5$ or less and by \eqref{eq:main_new:Gbound}, there are at most $4$ such triangular blocks. A bit of analysis shows that the maximum number of edges is achieved when the number of blocks of order $5$ is as large as possible.
\end{remark}

\section{Conclusions}

We note that the proof of Theorem~\ref{thm:main}, particularly Lemma~\ref{lem:partition}, can be rephrased in terms of a discharging argument.

We believe that our construction in Theorem~\ref{thm:construct} can be generalized to prove ${\rm ex}_{\mathcal{P}}(n,C_{\ell})$ for $\ell$ sufficiently large. That is, for certain values of $n$, we try to construct $G_0$, a plane graph with all faces of length $\ell+1$ with all vertices having degree $3$ or degree $2$.

If such a $G_0$ exists, then the number of degree-$2$ and degree-$3$ vertices are $\frac{(\ell-5)n+4(\ell+1)}{\ell-1}$ and $\frac{4(n-\ell-1)}{\ell-1}$, respectively. We could then apply steps similar to (1), (2), and (3) in the proof of Theorem~\ref{thm:construct} in that we add halving vertices and insert a graph $B_{\ell-1}$ (see Figure~\ref{fig:gadgets}) in place of vertices of degree $2$ and $3$. For the resulting graph $G$,
\begin{align*}
    v(G)    &=  v(G_0)+e(G_0)+(\ell-4)\frac{(\ell-5)n+4(\ell+1)}{\ell-1}+(\ell-5)\frac{4(n-\ell-1)}{\ell-1} \\
            &=  n+\frac{\ell+1}{\ell-1}(n-2)+\frac{(\ell^2-5\ell)n+2(\ell+1)}{\ell-1} \\
            &=  \frac{\ell^2-3\ell}{\ell-1}n+\frac{2(\ell+1)}{\ell} \\
    e(G)    &=  (3\ell-9)v(G_0)=(3\ell-9)n
\end{align*}

Therefore, $e(G)=\frac{3(\ell-1)}{\ell}v(G)-\frac{6(\ell+1)}{\ell}$. We conjecture that this is the maximum number of edges in a $C_{\ell}$-free planar graph.


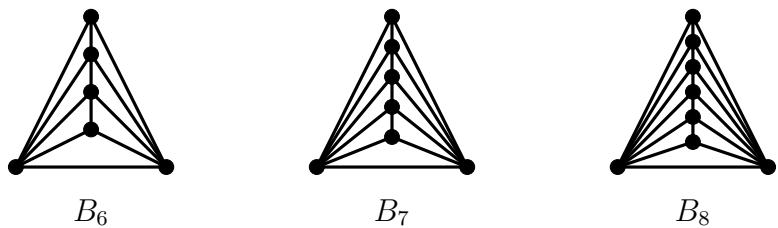
\begin{figure}[ht]
    \centering
    \begin{tikzpicture}[scale=0.2]
        \coordinate (x1) at (-5,0);
        \coordinate (x2) at (5,0);
        \coordinate (y3) at (0,10);
        \coordinate (y4) at (0,7.5);
        \coordinate (y5) at (0,5);
        \coordinate (y6) at (0,2.5);
        \draw[very thick] (x1) -- (x2);
        \draw[very thick] (x1) -- (y3) -- (x2);
        \draw[very thick] (x1) -- (y4) -- (x2);
        \draw[very thick] (x1) -- (y5) -- (x2);
        \draw[very thick] (x1) -- (y6) -- (x2);
        \draw[very thick] (y3) -- (y4) -- (y5) -- (y6);
        \draw[fill=black] (x1) circle(\circ);
        \draw[fill=black] (x2) circle(\circ);
        \draw[fill=black] (y3) circle(\circ);
        \draw[fill=black] (y4) circle(\circ);
        \draw[fill=black] (y5) circle(\circ);
        \draw[fill=black] (y6) circle(\circ);
        \node at (0,-3){$B_6$};
    \end{tikzpicture}
    \qquad\qquad
    \begin{tikzpicture}[scale=0.2]
        \coordinate (x1) at (-5,0);
        \coordinate (x2) at (5,0);
        \coordinate (y3) at (0,10);
        \coordinate (y4) at (0,8);
        \coordinate (y5) at (0,6);
        \coordinate (y6) at (0,4);
        \coordinate (y7) at (0,2);
        \draw[very thick] (x1) -- (x2);
        \draw[very thick] (x1) -- (y3) -- (x2);
        \draw[very thick] (x1) -- (y4) -- (x2);
        \draw[very thick] (x1) -- (y5) -- (x2);
        \draw[very thick] (x1) -- (y6) -- (x2);
        \draw[very thick] (x1) -- (y7) -- (x2);
        \draw[very thick] (y3) -- (y4) -- (y5) -- (y6) -- (y7);
        \draw[fill=black] (x1) circle(\circ);
        \draw[fill=black] (x2) circle(\circ);
        \draw[fill=black] (y3) circle(\circ);
        \draw[fill=black] (y4) circle(\circ);
        \draw[fill=black] (y5) circle(\circ);
        \draw[fill=black] (y6) circle(\circ);
        \draw[fill=black] (y7) circle(\circ);
        \node at (0,-3){$B_7$};
    \end{tikzpicture}
    \qquad\qquad
    \begin{tikzpicture}[scale=0.2]
        \coordinate (x1) at (-5,0);
        \coordinate (x2) at (5,0);
        \coordinate (y3) at (0,10);
        \coordinate (y4) at (0,25/3);
        \coordinate (y5) at (0,20/3);
        \coordinate (y6) at (0,5);
        \coordinate (y7) at (0,10/3);
        \coordinate (y8) at (0,5/3);
        \draw[very thick] (x1) -- (x2);
        \draw[very thick] (x1) -- (y3) -- (x2);
        \draw[very thick] (x1) -- (y4) -- (x2);
        \draw[very thick] (x1) -- (y5) -- (x2);
        \draw[very thick] (x1) -- (y6) -- (x2);
        \draw[very thick] (x1) -- (y7) -- (x2);
        \draw[very thick] (x1) -- (y8) -- (x2);
        \draw[very thick] (y3) -- (y4) -- (y5) -- (y6) -- (y7) -- (y8);
        \draw[fill=black] (x1) circle(\circ);
        \draw[fill=black] (x2) circle(\circ);
        \draw[fill=black] (y3) circle(\circ);
        \draw[fill=black] (y4) circle(\circ);
        \draw[fill=black] (y5) circle(\circ);
        \draw[fill=black] (y6) circle(\circ);
        \draw[fill=black] (y7) circle(\circ);
        \draw[fill=black] (y8) circle(\circ);
       \node at (0,-3){$B_8$};
    \end{tikzpicture}
    \caption{$B_{\ell-1}$ is used in the construction of a $C_{\ell}$-free graph.}
    \label{fig:gadgets}
\end{figure}

\begin{conjecture}
    Let $G$ be an $n$-vertex $C_{\ell}$-free plane graph ($\ell\geq 7$), then there exists an integer $N_0>0$, such that when $n\geq N_0$, $e(G)\leq \frac{3(\ell-1)}{\ell}n-\frac{6(\ell+1)}{\ell}$.
\end{conjecture}

\section{Acknowledgements}
Gy\H ori's research was partially supported by the National Research, Development and Innovation Office NKFIH, grants  K132696, K116769, and K126853.
Martin's research was partially supported by Simons Foundation Collaboration Grant \#353292 and by the J. William Fulbright Educational Exchange Program.


\newpage
\appendix
\section{Tables}
\label{appendix}

The following tables give a summary of the results from Lemmas~\ref{lem:easyblock}, \ref{lem:B5d}, and~\ref{lem:B4b}.

A red edge incident to a vertex of a triangular-block indicates the corresponding vertex is a junction vertex. Moreover, if a vertex has only one red edge, it is to indicate the vertex is shared in at least two triangular-blocks. Whereas if a vertex has two red edges, it means that the vertex is shared in at least three blocks.

A pair of blue edges indicates the boundary of a $4$-face.

\begin{table}[ht]
    \renewcommand\arraystretch{2}\par
    \centering
    \newlength{\casebox}
    \setlength{\casebox}{1.8cm}
    \begin{tabular}{|c|c|c|c|c|c|c|}\hline
        \makebox[\casebox][c]{Case} &
        \makebox[0.7cm]{$B$} &
        \makebox[3.0cm]{Diagram} &
        \makebox[1.8cm]{$f(B)\leq$} &
        \makebox[1.4cm]{$n(B)\leq$} &
        \makebox[1.4cm]{$e(B)=$} &
        \makebox[2.9cm]{$7f+2n-5e\leq$} \\ \hline \hline
    \parbox[c]{\casebox}{\centering Lemma~\ref{lem:easyblock} \\ 1\ref{case1a}} &
    $B_{5,a}$   &
    \parbox[c]{1em}{%
        \renewcommand\arraystretch{2}\par%
        \centerline{%
            \begin{tikzpicture}[scale=0.2]
                \coordinate (x1) at (-3,0);
                \coordinate (x2) at (0,5);
                \coordinate (x3) at (3,0);
                \coordinate (l1) at (-6,-1.5);
                \coordinate (l2) at (0,8);
                \coordinate (l3) at (6,-1.5);
                \draw[red,very thick] (l1) -- (x1);
                \draw[red,very thick] (l2) -- (x2);
                \draw[red,very thick] (l3) -- (x3);
                \draw[very thick,fill=gray!20] (x1) -- (x2) -- (x3) -- (x1);
                \draw[fill=black] (x1) circle(\circ);
                \draw[fill=black] (x2) circle(\circ);
                \draw[fill=black] (x3) circle(\circ);
                \draw[fill=black] (l1) circle(\circ);
                \draw[fill=black] (l2) circle(\circ);
                \draw[fill=black] (l3) circle(\circ);
                \node at (0,1.5) {$K_5^-$};
                \node at (-6.5,-2) {};
                \node at (6.5,8.5) {};
            \end{tikzpicture}}} &
    $5 + \dfrac{3}{7}$ & $2 + \dfrac{3}{2}$ & 9 & 0 \\ \hline
    \parbox[c]{\casebox}{\centering Lemma~\ref{lem:easyblock} \\ 1\ref{case1b}} &
    $B_{5,a}$    &
    \parbox[c]{1em}{%
        \renewcommand\arraystretch{2}\par%
        \centerline{%
            \begin{tikzpicture}[scale=0.2]
                \coordinate (x1) at (-3,0);
                \coordinate (x2) at (0,5);
                \coordinate (x3) at (3,0);
                \coordinate (l1) at (-6,-1.5);
                \coordinate (l3) at (6,-1.5);
                \draw[red,very thick] (l1) -- (x1);
                \draw[red,very thick] (l3) -- (x3);
                \draw[very thick,fill=gray!20] (x1) -- (x2) -- (x3) -- (x1);
                \draw[fill=black] (x1) circle(\circ);
                \draw[fill=black] (x2) circle(\circ);
                \draw[fill=black] (x3) circle(\circ);
                \draw[fill=black] (l1) circle(\circ);
                \draw[fill=black] (l3) circle(\circ);
                \node at (0,1.5) {$K_5^-$};
                \node at (-6.5,-2) {};
                \node at (6.5,5.5) {};
            \end{tikzpicture}}} &
    $5 + \dfrac{2}{7}$ & $3 + \dfrac{2}{2}$ & 9 & 0 \\ \hline
    \parbox[c]{\casebox}{\centering Lemma~\ref{lem:easyblock} \\ 2\ref{case2b}} &
    $B_{5,b}$   &
    \parbox[c]{1em}{%
        \renewcommand\arraystretch{2}\par%
        \centerline{%
            \begin{tikzpicture}[scale=0.2]
                \coordinate (x1) at (-3,6);
                \coordinate (x2) at (3,6);
                \coordinate (x3) at (3,0);
                \coordinate (x4) at (-3,0);
                \coordinate (x5) at (0,3);
                \draw[very thick] (x1) -- (x2) -- (x3) -- (x4) -- (x1);
                \draw[very thick] (x1) -- (x5);
                \draw[very thick] (x2) -- (x5);
                \draw[very thick] (x3) -- (x5);
                \draw[very thick] (x4) -- (x5);
                \draw[very thick] (x5) -- (x5);
                \draw[fill=black] (x1) circle(\circ);
                \draw[fill=black] (x2) circle(\circ);
                \draw[fill=black] (x3) circle(\circ);
                \draw[fill=black] (x4) circle(\circ);
                \draw[fill=black] (x5) circle(\circ);
                \node at (-3.5,-0.5) {};
                \node at (3.5,6.5) {};
            \end{tikzpicture}}} &
    $4 + \dfrac{4}{7}$ & $3 + \dfrac{2}{2}$ & 8 & 0 \\ \hline
    \parbox[c]{\casebox}{\centering Lemma~\ref{lem:easyblock} \\ 2\ref{case2c}} &
    $B_{5,c}$   &
    \parbox[c]{1em}{%
        \renewcommand\arraystretch{2}\par%
        \centerline{%
            \begin{tikzpicture}[scale=0.2]
                \coordinate (x1) at (-3,6);
                \coordinate (x2) at (3,6);
                \coordinate (x3) at (5,3);
                \coordinate (x4) at (3,0);
                \coordinate (x5) at (-3,0);
                \draw[very thick] (x1) -- (x2) -- (x3) -- (x4) -- (x5) -- (x1);
                \draw[very thick] (x5) -- (x2) -- (x4);
                \draw[fill=black] (x1) circle(\circ);
                \draw[fill=black] (x2) circle(\circ);
                \draw[fill=black] (x3) circle(\circ);
                \draw[fill=black] (x4) circle(\circ);
                \draw[fill=black] (x5) circle(\circ);
                \node at (-3.5,-0.5) {};
                \node at (5.5,6.5) {};
            \end{tikzpicture}}} &
    $3+ \dfrac{5}{7}$ & $3 + \dfrac{2}{2}$ & 7 & $-1$ \\ \hline
    \parbox[c]{\casebox}{\centering Lemma~\ref{lem:B5d} \\ \ref{case5a}} &
    $B_{5,d}$   &
    \parbox[c]{1em}{%
        \renewcommand\arraystretch{2}\par%
        \centerline{%
            \begin{tikzpicture}[scale=0.2]
                \coordinate (x1) at (0,8);
                \coordinate (x2) at (4,4);
                \coordinate (x3) at (2,4);
                \coordinate (x4) at (0,0);
                \coordinate (x5) at (-4,4);
                \coordinate (l2) at (7,4);
                \coordinate (l5) at (-7,4);
                \draw[red,very thick] (l2) -- (x2);
                \draw[red,very thick] (l5) -- (x5);
                \draw[very thick] (x1) -- (x2) -- (x4) -- (x5) -- (x1);
                \draw[very thick] (x1) -- (x4);
                \draw[very thick] (x1) -- (x3) -- (x4);
                \draw[very thick] (x2) -- (x3);
                \draw[fill=black] (x1) circle(\circ);
                \draw[fill=black] (x2) circle(\circ);
                \draw[fill=black] (x3) circle(\circ);
                \draw[fill=black] (x4) circle(\circ);
                \draw[fill=black] (x5) circle(\circ);
                \draw[fill=black] (l2) circle(\circ);
                \draw[fill=black] (l5) circle(\circ);
                \node at (-7.5,-0.5) {};
                \node at (7.5,8.5) {};
            \end{tikzpicture}}} &
    $4 + \dfrac{4}{7}$ & $3 + \dfrac{2}{2}$ & 8 & $0$ \\ \hline
    \parbox[c]{\casebox}{\centering Lemma~\ref{lem:B5d} \\ \ref{case5b}} &
    $B_{5,d}$   &
    \parbox[c]{1em}{%
        \renewcommand\arraystretch{2}\par%
        \centerline{%
            \begin{tikzpicture}[scale=0.2]
                \coordinate (x1) at (0,8);
                \coordinate (x2) at (4,4);
                \coordinate (x3) at (2,4);
                \coordinate (x4) at (0,0);
                \coordinate (x5) at (-4,4);
                \coordinate (l4) at (4,0);
                \coordinate (l5) at (-7,4);
                \draw[red,very thick] (l5) -- (x5);
                \draw[red,very thick] (l4) -- (x4);
                \draw[very thick] (x1) -- (x2) -- (x4) -- (x5) -- (x1);
                \draw[very thick] (x1) -- (x4);
                \draw[very thick] (x1) -- (x3) -- (x4);
                \draw[very thick] (x2) -- (x3);
                \draw[fill=black] (x1) circle(\circ);
                \draw[fill=black] (x2) circle(\circ);
                \draw[fill=black] (x3) circle(\circ);
                \draw[fill=black] (x4) circle(\circ);
                \draw[fill=black] (x5) circle(\circ);
                \draw[fill=black] (l5) circle(\circ);
                \draw[fill=black] (l4) circle(\circ);
                \node at (-7.5,-0.5) {};
                \node at (7.5,8.5) {};
            \end{tikzpicture}}} &
    $4 + \dfrac{4}{7}$ & $3 + \dfrac{2}{2}$ & 8 & $0$ \\ \hline
    \parbox[c]{\casebox}{\centering Lemma~\ref{lem:B5d} \\ \ref{case5c}} &
    $B_{5,d}$   &
    \parbox[c]{1em}{%
        \renewcommand\arraystretch{2}\par%
        \centerline{%
            \begin{tikzpicture}[scale=0.2]
                \coordinate (x1) at (0,8);
                \coordinate (x2) at (4,4);
                \coordinate (x3) at (2,4);
                \coordinate (x4) at (0,0);
                \coordinate (x5) at (-4,4);
                \coordinate (l2) at (7,4);
                \coordinate (l5) at (-7,4);
                \draw[red,very thick] (l5) -- (x5);
                \draw[blue,very thick] (x1) -- (l2) -- (x4);
                \draw[very thick] (x1) -- (x2) -- (x4) -- (x5) -- (x1);
                \draw[very thick] (x1) -- (x4);
                \draw[very thick] (x1) -- (x3) -- (x4);
                \draw[very thick] (x2) -- (x3);
                \draw[fill=black] (x1) circle(\circ);
                \draw[fill=black] (x2) circle(\circ);
                \draw[fill=black] (x3) circle(\circ);
                \draw[fill=black] (x4) circle(\circ);
                \draw[fill=black] (x5) circle(\circ);
                \draw[fill=black] (l2) circle(\circ);
                \draw[fill=black] (l5) circle(\circ);
                \node at (-7.5,-0.5) {};
                \node at (7.5,8.5) {};
            \end{tikzpicture}}} &
    $4 + \dfrac{2}{4} + \dfrac{2}{7}$ & $2 + \dfrac{3}{2}$ & 8 & $\dfrac{1}{2}$~\redstar \\ \hline
    \end{tabular}
    \caption{All possible $B_5$ blocks in $G$ and the estimation of $7f(B)+2n(B)-5e(B)$.}
    \label{tab:5block}
\end{table}

\begin{table}
    \renewcommand\arraystretch{2}\par
    \centering
    \setlength{\casebox}{1.8cm}
    \begin{tabular}{|c|c|c|c|c|c|c|}\hline
        \makebox[\casebox][c]{Case} &
        \makebox[0.7cm]{$B$} &
        \makebox[3.0cm]{Diagram} &
        \makebox[1.8cm]{$f(B)\leq$} &
        \makebox[1.4cm]{$n(B)\leq$} &
        \makebox[1.4cm]{$e(B)=$} &
        \makebox[2.9cm]{$7f+2n-5e\leq$} \\ \hline \hline
        \parbox[c]{\casebox}{\centering Lemma~\ref{lem:easyblock} \\ 2\ref{case2a}} &
        $B_{4,a}$   &
        \parbox[c]{1em}{%
            \renewcommand\arraystretch{2}\par%
            \centerline{%
                \begin{tikzpicture}[scale=0.2]
                    \coordinate (x1) at (-4,0);
                    \coordinate (x2) at (0,7);
                    \coordinate (x3) at (4,0);
                    \coordinate (x4) at (0,3);
                    \draw[very thick] (x1) -- (x2) -- (x3) -- (x1);
                    \draw[very thick] (x1) -- (x4) -- (x3);
                    \draw[very thick] (x2) -- (x4);
                    \draw[fill=black] (x1) circle(\circ);
                    \draw[fill=black] (x2) circle(\circ);
                    \draw[fill=black] (x3) circle(\circ);
                    \draw[fill=black] (x4) circle(\circ);
                    \node at (-4.5,-0.5) {};
                    \node at (4.5,7.5) {};
                \end{tikzpicture}}} &
        $3 + \dfrac{3}{7}$ & $2 + \dfrac{2}{2}$ & 6 & 0 \\ \hline
        \parbox[c]{\casebox}{\centering Lemma~\ref{lem:B4b} \\ \ref{case6a}} &
        $B_{4,b}$   &
        \parbox[c]{1em}{%
            \renewcommand\arraystretch{2}\par%
            \centerline{%
                \begin{tikzpicture}[scale=0.2]
                    \coordinate (x1) at (0,8);
                    \coordinate (x2) at (4,4);
                    \coordinate (x3) at (0,0);
                    \coordinate (x4) at (-4,4);
                    \coordinate (l2) at (7,4);
                    \coordinate (l4) at (-7,4);
                    \draw[red,very thick] (l4) -- (x4);
                    \draw[blue,very thick] (x1) -- (l2) -- (x3);
                    \draw[very thick] (x1) -- (x2) -- (x3) --    (x4) -- (x1);
                    \draw[very thick] (x2) -- (x4);
                    \draw[fill=black] (x1) circle(\circ);
                    \draw[fill=black] (x2) circle(\circ);
                    \draw[fill=black] (x3) circle(\circ);
                    \draw[fill=black] (x4) circle(\circ);
                    \draw[fill=black] (l2) circle(\circ);
                    \draw[fill=black] (l4) circle(\circ);
                    \node at (-7.5,-0.5) {};
                    \node at (7.5,8.5) {};
                \end{tikzpicture}}} &
        $2 + \dfrac{2}{4} + \dfrac{2}{7}$ & $1 + \dfrac{3}{2}$ & 5 & $-\dfrac{1}{2}$ \\ \hline
        \parbox[c]{\casebox}{\centering Lemma~\ref{lem:B4b} \\ \ref{case6b}\ref{case6bi}} &
        $B_{4,b}$   &
        \parbox[c]{1em}{%
            \renewcommand\arraystretch{2}\par%
            \centerline{%
                \begin{tikzpicture}[scale=0.2]
                    \coordinate (x1) at (0,8);
                    \coordinate (x2) at (4,4);
                    \coordinate (x3) at (0,0);
                    \coordinate (x4) at (-4,4);
                    \coordinate (l3) at (4,0);
                    \coordinate (l1) at (-4,8);
                    \draw[red,very thick] (l3) -- (x3);
                    \draw[red,very thick] (l1) -- (x1);
                    \draw[very thick] (x1) -- (x2) -- (x3) -- (x4) -- (x1);
                    \draw[very thick] (x2) -- (x4);
                    \draw[fill=black] (x1) circle(\circ);
                    \draw[fill=black] (x2) circle(\circ);
                    \draw[fill=black] (x3) circle(\circ);
                    \draw[fill=black] (x4) circle(\circ);
                    \draw[fill=black] (l3) circle(\circ);
                    \draw[fill=black] (l1) circle(\circ);
                    \node at (-7.5,-0.5) {};
                    \node at (7.5,8.5) {};
                \end{tikzpicture}}} &
        $2 + \dfrac{4}{7}$ & $2 + \dfrac{2}{2}$ & 5 & $-1$ \\ \hline
        \parbox[c]{\casebox}{\centering Lemma~\ref{lem:B4b} \\ \ref{case6b}\ref{case6bii}} &
        $B_{4,b}$   &
        \parbox[c]{1em}{%
            \renewcommand\arraystretch{2}\par%
            \centerline{%
                \begin{tikzpicture}[scale=0.2]
                    \coordinate (x1) at (0,8);
                    \coordinate (x2) at (4,4);
                    \coordinate (x3) at (0,0);
                    \coordinate (x4) at (-4,4);
                    \coordinate (l2) at (7,4);
                    \coordinate (l4) at (-7,4);
                    \coordinate (l1) at (-4,8);
                    \draw[red,very thick] (l1) -- (x1);
                    \draw[blue,very thick] (x1) -- (l2) -- (x3);
                    \draw[very thick] (x1) -- (x2) -- (x3) -- (x4) -- (x1);
                    \draw[very thick] (x2) -- (x4);
                    \draw[fill=black] (x1) circle(\circ);
                    \draw[fill=black] (x2) circle(\circ);
                    \draw[fill=black] (x3) circle(\circ);
                    \draw[fill=black] (x4) circle(\circ);
                    \draw[fill=black] (l2) circle(\circ);
                    \draw[fill=black] (l1) circle(\circ);
                   \node at (-7.5,-0.5) {};
                    \node at (7.5,8.5) {};
                \end{tikzpicture}}} &
        $2 + \dfrac{2}{4} + \dfrac{2}{7}$ & $2 + \dfrac{1}{3} + \dfrac{1}{2}$ & 5 & $\dfrac{1}{6}$~\redstar \\ \hline
        \parbox[c]{\casebox}{\centering Lemma~\ref{lem:B4b} \\ \ref{case6b}\ref{case6biii}} &
        $B_{4,b}$   &
        \parbox[c]{1em}{%
            \renewcommand\arraystretch{2}\par%
            \centerline{%
                \begin{tikzpicture}[scale=0.2]
                    \coordinate (x1) at (0,8);
                    \coordinate (x2) at (4,4);
                    \coordinate (x3) at (0,0);
                    \coordinate (x4) at (-4,4);
                    \coordinate (l2) at (7,4);
                    \coordinate (l4) at (-7,4);
                    \draw[blue,very thick] (x1) -- (l2) -- (x3);
                    \draw[blue,very thick] (x1) -- (l4) -- (x3);
                    \draw[very thick] (x1) -- (x2) -- (x3) -- (x4) -- (x1);
                    \draw[very thick] (x2) -- (x4);
                    \draw[fill=black] (x1) circle(\circ);
                    \draw[fill=black] (x2) circle(\circ);
                    \draw[fill=black] (x3) circle(\circ);
                    \draw[fill=black] (x4) circle(\circ);
                    \draw[fill=black] (l2) circle(\circ);
                    \draw[fill=black] (l4) circle(\circ);
                    \node at (-7.5,-0.5) {};
                    \node at (7.5,8.5) {};
                \end{tikzpicture}}} &
        $2 + \dfrac{2}{4} + \dfrac{2}{4}$ & $2 + \dfrac{2}{3}$ & 5 & $\dfrac{4}{3}$~\redstar \\ \hline
        \parbox[c]{\casebox}{\centering Lemma~\ref{lem:B4b} \\ 3\ref{case3a}} &
        $B_3$   &
        \parbox[c]{1em}{%
            \renewcommand\arraystretch{2}\par%
            \centerline{%
                \begin{tikzpicture}[scale=0.2]
                    \coordinate (x1) at (-3,0);
                    \coordinate (x2) at (0,5);
                    \coordinate (x3) at (3,0);
                    \coordinate (l1) at (-6,-1.5);
                    \coordinate (l2) at (0,8);
                    \coordinate (l3) at (6,-1.5);
                    \draw[red,very thick] (l1) -- (x1);
                    \draw[red,very thick] (l2) -- (x2);
                    \draw[red,very thick] (l3) -- (x3);
                    \draw[very thick] (x1) -- (x2) -- (x3) -- (x1);
                    \draw[fill=black] (x1) circle(\circ);
                    \draw[fill=black] (x2) circle(\circ);
                    \draw[fill=black] (x3) circle(\circ);
                    \draw[fill=black] (l1) circle(\circ);
                    \draw[fill=black] (l2) circle(\circ);
                    \draw[fill=black] (l3) circle(\circ);
                    \node at (-6.5,-2) {};
                    \node at (6.5,8.5) {};
                \end{tikzpicture}}} &
        $1 + \dfrac{2}{7} + \dfrac{1}{4}$ & $\dfrac{3}{2}$ & 3 & $-\dfrac{5}{4}$ \\ \hline
        \parbox[c]{\casebox}{\centering Lemma~\ref{lem:B4b} \\ 3\ref{case3b}}  &
        $B_3$   &
        \parbox[c]{1em}{%
            \renewcommand\arraystretch{2}\par%
            \centerline{%
                \begin{tikzpicture}[scale=0.2]
                    \coordinate (x1) at (-3,0);
                    \coordinate (x2) at (0,5);
                    \coordinate (x3) at (3,0);
                    \coordinate (l1) at (-6,-1.5);
                    \coordinate (l2) at (-2,7.5);
                    \coordinate (l4) at (2,7.5);
                    \coordinate (l3) at (6,-1.5);
                    \draw[red,very thick] (l1) -- (x1);
                    \draw[red,very thick] (l2) -- (x2);
                    \draw[red,very thick] (l4) -- (x2);
                    \draw[red,very thick] (l3) -- (x3);
                    \draw[very thick] (x1) -- (x2) -- (x3) -- (x1);
                    \draw[fill=black] (x1) circle(\circ);
                    \draw[fill=black] (x2) circle(\circ);
                    \draw[fill=black] (x3) circle(\circ);
                    \draw[fill=black] (l1) circle(\circ);
                    \draw[fill=black] (l2) circle(\circ);
                    \draw[fill=black] (l4) circle(\circ);
                    \draw[fill=black] (l3) circle(\circ);
                    \node at (-6.5,-2) {};
                    \node at (6.5,8.0) {};
                \end{tikzpicture}}} &
        $1 + \dfrac{3}{4}$ & $\dfrac{2}{2} + \dfrac{1}{3}$ & 3 & $-\dfrac{1}{12}$ \\ \hline
        \parbox[c]{\casebox}{\centering Lemma~\ref{lem:B4b} \\ 4\ref{case4a}} &
        $B_2$   &
        \parbox[c]{1em}{%
            \renewcommand\arraystretch{2}\par%
            \centerline{%
                \begin{tikzpicture}[scale=0.2]
                    \coordinate (x1) at (-3,0);
                    \coordinate (x3) at (3,0);
                    \coordinate (l1) at (-6,0);
                    \coordinate (l3) at (6,0);
                    \draw[red,very thick] (l1) -- (x1);
                    \draw[red,very thick] (l3) -- (x3);
                    \draw[very thick] (x1) -- (x3);
                    \draw[fill=black] (x1) circle(\circ);
                    \draw[fill=black] (x3) circle(\circ);
                    \draw[fill=black] (l1) circle(\circ);
                    \draw[fill=black] (l3) circle(\circ);
                    \node at (-6.5,-2.5) {};
                    \node at (6.5,2.5) {};
                \end{tikzpicture}}} &
        $\dfrac{1}{4} + \dfrac{1}{7}$ & $\dfrac{2}{2}$ & 1 & $-\dfrac{1}{4}$ \\ \hline
        \parbox[c]{\casebox}{\centering Lemma~\ref{lem:B4b} \\ 4\ref{case4b}} &
        $B_2$   &
        \parbox[c]{1em}{%
            \renewcommand\arraystretch{2}\par%
            \centerline{%
                \begin{tikzpicture}[scale=0.2]
                    \coordinate (x1) at (-3,0);
                    \coordinate (x3) at (3,0);
                    \coordinate (l1) at (-6,0);
                    \coordinate (l3) at (6,2);
                    \coordinate (l4) at (6,-2);
                    \draw[red,very thick] (l1) -- (x1);
                    \draw[red,very thick] (l3) -- (x3);
                    \draw[red,very thick] (l4) -- (x3);
                    \draw[very thick] (x1) -- (x3);
                    \draw[fill=black] (x1) circle(\circ);
                    \draw[fill=black] (x3) circle(\circ);
                    \draw[fill=black] (l1) circle(\circ);
                    \draw[fill=black] (l3) circle(\circ);
                    \draw[fill=black] (l4) circle(\circ);
                    \node at (-6.5,-2.5) {};
                    \node at (6.5,2.5) {};
                \end{tikzpicture}}} &
        $\dfrac{1}{4} + \dfrac{1}{7}$ & $\dfrac{1}{2} + \dfrac{1}{3}$ & 1 & $-\dfrac{7}{12}$ \\ \hline
        \parbox[c]{\casebox}{\centering Lemma~\ref{lem:B4b} \\ 4\ref{case4c}} &
        $B_2$   &
        \parbox[c]{1em}{%
            \renewcommand\arraystretch{2}\par%
            \centerline{%
                \begin{tikzpicture}[scale=0.2]
                    \coordinate (x1) at (-3,0);
                    \coordinate (x3) at (3,0);
                    \coordinate (l1) at (-6,2);
                    \coordinate (l2) at (-6,-2);
                    \coordinate (l3) at (6,2);
                    \coordinate (l4) at (6,-2);
                    \draw[red,very thick] (l1) -- (x1);
                    \draw[red,very thick] (l2) -- (x1);
                    \draw[red,very thick] (l3) -- (x3);
                    \draw[red,very thick] (l4) -- (x3);
                    \draw[very thick] (x1) -- (x3);
                    \draw[fill=black] (x1) circle(\circ);
                    \draw[fill=black] (x3) circle(\circ);
                    \draw[fill=black] (l1) circle(\circ);
                    \draw[fill=black] (l2) circle(\circ);
                    \draw[fill=black] (l3) circle(\circ);
                    \draw[fill=black] (l4) circle(\circ);
                    \node at (-6.5,-2.5) {};
                    \node at (6.5,2.5) {};
                \end{tikzpicture}}} &
        $\dfrac{1}{4} + \dfrac{1}{5}$ & $\dfrac{2}{3}$ & 1 & $-\dfrac{31}{60}$ \\ \hline
    \end{tabular}
    \caption{All possible $B_4,B_3$ and $B_2$ blocks in $G$ and the estimate of $7f(B)+2n(B)-5e(B)$.}
    \label{tab:432block}
\end{table}

\end{document}